\documentclass[11pt, oneside, 
a4paper]{article}
\usepackage[ansinew]{inputenc}
\usepackage[english]{babel}
\usepackage{amsbsy,amscd,amsfonts,amsmath,amsopn,amssymb,amstext,amsthm,amsxtra,color,fancybox,float,graphicx,latexsym,srcltx,times,tikz,url}
\newcommand{\nega}{\hspace{-1em}}

\newcommand{\sucen}[2]{{#1_{1},\hdots,#1_{#2}}}

\newcommand{\suce}[1]{{#1_{1},#1_{2},#1_{3}}}

\renewcommand{\P}{\mathbb{P}}
\newcommand{\T}{\mathbb{T}}

\newcommand{\TP}{\mathbb{TP}^3}

\newcommand{\colu}[4]{\left[\begin{array}{c}#1\\#2\\#3\\#4\\
\end{array}\right]}
\newcommand{\colur}[4]{\left[\begin{array}{r}#1\\#2\\#3\\#4\\
\end{array}\right]}

\newcommand{\m}{\medskip}

\newcommand{\N}{\mathbb{N}}

\newcommand{\R}{\mathbb{R}}

\newcommand{\SSS}{\mathcal{S}}

\renewcommand{\int}{\operatorname{int}}

\newcommand{\adiff}{\operatorname{adiff}}

\newcommand{\col}{\operatorname{col}}

\newcommand{\card}{\operatorname{card}}

\newcommand{\trop}{\operatorname{trop}}
\newcommand{\rk}{\operatorname{rk}}

\newcommand{\diag}{\operatorname{diag}}
\newcommand{\spann}{\operatorname{span}}
\newcommand{\tdist}{\operatorname{tdist}}

\newcommand{\dist}{\operatorname{dist}}

\newcommand{\cnst}{\operatorname{cnst}}

\newtheorem{thm}{Theorem}
\newtheorem{lem}[thm]{Lemma}

\newtheorem{cor}[thm]{Corollary}
\newtheorem{ex}[thm]{Example}
\newtheorem{exr}[thm]{Exercises}
\newtheorem{rem}[thm]{Remark}

\title{Six combinatorial classes of maximal convex tropical tetrahedra}

\author{A. Jim\'{e}nez 
\and
M.J. de la Puente\thanks{Partially supported by UCM research group 910444. Corresponding author}}

\date{} 

\begin{document}
\maketitle



\centerline{\textbf{Short title}: Six classes of maximal convex tropical tetrahedra}

\begin{abstract} In this  paper we bring together tropical linear algebra  and convex 3--dimensional bodies.
We show how certain convex 3--dimensional bodies having 20 vertices and 12 facets can be encoded in a $4\times 4$ integer  zero--diagonal matrix  $A$.
A   tropical tetrahedron  is the set of points in $\R^3$ tropically spanned by  four  given tropically non--coplanar points. It is a near--miss Johnson solid. The coordinates of the points are  arranged as the columns of a $4\times 4$ real matrix $A$ and the tetrahedron is denoted $\spann(A)$.
We study tropical tetrahedra which are convex  and  maximal, computing the extremals of $\spann(A)$ and the length (tropical or Euclidean) of its edges.
Then, we classify convex  maximal tropical tetrahedra, combinatorially.   There are six classes, up to symmetry and chirality. Only one class contains symmetric solids and only one contains chiral ones. In the way, we show that the combinatorial type of the regular dodecahedron does not occur here. We also prove that convex maximal tropical tetrahedra are not vertex--transitive, in general. We give families of examples, circulant matrices providing examples for two classes. A crucial role   is played  by the $2\times 2$ minors of $A$.
\end{abstract}

\section{Introduction} \label{sec:intro}

Here $\oplus=\max$, $\odot=+$ are the tropical operations: addition and multiplication.
In classical mathematics, a tetrahedron is the span of four non--coplanar points in 3--dimensional space. We have a choice to make: affine
or projective geometry. In tropical mathematics, a  tetrahedron should be the  tropical
span of four tropically  non--coplanar points in 3--dimensional space. It is known that this set is neither pure--dimensional nor convex, in general.

Here we study tropical tetrahedra which are both convex and maximal. The coordinates of the points are  arranged as the columns of a $4\times 4$ real matrix $A$ and the tetrahedron is denoted $\spann(A)$.
A leading role   is played  by the $2\times 2$ minors of $A$. Firstly, because they provide the tropical length  of the edges  of $\spann(A)$
(theorem \ref{thm:adiffs}). Secondly, because they characterize maximality of
$\spann(A)$ (lemma \ref{lem:maximal}).
The $2\times 2$ minors of $A$ also yield the types of the tropical
lines $L_{ij}:=L(\col(A,i),\col(A,j))$ and, with the type information we are able compute the coordinates of all the extremals of $\spann(A)$. 

\m
In section \ref{sec:tetra}, we show that the f--vector (i.e., vector whose components are the number of vertices, edges and facets) of  a maximal convex $\spann(A)$ is  $(20,30,12)$, with  polygon--vector  (defined in p.\pageref{dfn:polygonvector}) $(0,f_4,f_5,f_6)$, where $12=f_4+f_5+f_6$.
In fact, the polygon--vector can only be $(0,2,8,2)$, $(0,3,6,3)$ or $(0,4,4,4)$ and so the combinatorial type of the regular dodecahedron does not arise in this setting.

Our classification of the combinatorial types of $\spann(A)$ depends on the type--vector (defined in p. \pageref{dfn:type--vector})
and number and adjacency of the hexagonal facets in $\spann(A)$ (knowing the polygon--vector is not enough!). The type--vector can be  $(2,2,2)$, $(3,2,1)$, $(3,3,0)$, $(4,1,1)$ or $(4,2,0)$, up to a permutation.
\textbf{Six} different  classes  exist, although just five classes have been announced in \cite{Joswig_K}.


\m

This paper arose from reading  \cite{Joswig_K} and also \cite{Develin_Sturm}.
A classification of tropical tetrahedra (with a less restrictive definition than ours) has been announced in \cite{Joswig_K}, containing  \textbf{five} combinatorial types.
In
\cite{Joswig_K},  a very brief general description, plus five  matrices and five
pictures are provided in only half a page.  There, the f--vector  is  claimed to be $(20,30,12)$  for every such $\spann(A)$. Notice that this is
the f--vector of the \textbf{regular dodecahedron} $\cal D$. We have two concerns about the five--item list in \cite{Joswig_K}. First, no description of each particular
  item is given.  Which are the polygon--vectors, i.e., which polygons appear as facets and how many of each, for each item? Just looking at \cite{Joswig_K}, we cannot tell.
  Moreover, we cannot answer another natural question such as: does any item in the list
have the combinatorial type of $\cal D$? These  questions remain unanswered after looking at   the web page by the same authors,
http://wwwopt.mathematik.tu-darmstadt.de/~kulas/Polytrope.html, containing additional information.
Our second concern is that the list  in \cite{Joswig_K} misses one  class. It misses class \ref{class:third}, having type--vector  $(3,2,1)$ and  two adjacent hexagons (the polygon--vector is $(0,2,8,2)$); see examples \ref{ex:B15_y_B12}. This class is very important, since it shows that $\spann(A)$ is not vertex--transitive, in general.

\m
The new material is found in section \ref{sec:tetra}. The previous sections are introductory. We have tried to make the paper self--contained. We have computed a lot of examples, thoroughly, for the benefit of the reader.

\m
An earlier extended version of this paper was uploaded on 18/05/2012 in arXiv 1205.4162 with the title "Characterizing the convexity of the $n$--dimensional tropical simplex and the six convex classes in $\R^3$".


\begin{figure}[h]
 \centering
  \includegraphics[keepaspectratio,width=14cm]{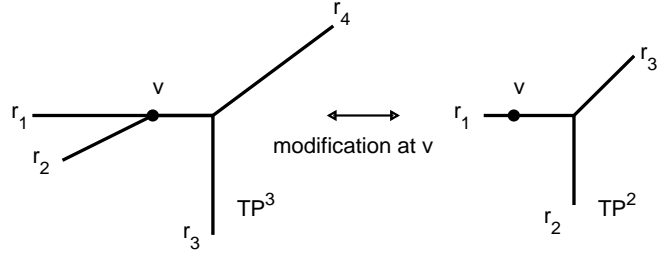}\\
  \caption{Modification of a tropical line at $v$: going from $\TP$ to $\T\P^2$ and back.}
  \label{figure_01}
\end{figure}

\section{Background and notations}\label{sec:background}
We will work  over  $\T:=(\R,\oplus,\odot)$, where $\oplus=\max$ is tropical addition and $\odot=+$ is tropical multiplication.
For $n\in \N$, $[n]$ denotes the set $\{1,2,\ldots,n\}$. For $n,m\in\N$, $\R^{m\times n}$ denotes the set of real matrices having $m$ rows and $n$ columns.
Define tropical sum and product of matrices following the same rules of classical linear algebra, but replacing addition (multiplication) by tropical addition (multiplication).
We will never use classical sum or multiplication of matrices, in this note. $A\odot B$ will be written $AB$, for simplicity,
 for matrices $A,B$.

\m
The \textbf{tropical determinant} (also called tropical permanent) of   $A=(a_{ij})\in\R^{n\times n}$
 is defined as $$|{A}|_{\trop}=\max_{\sigma\in S_n}\{a_{1\sigma(1)}+a_{2\sigma(2)}+\cdots+a_{n\sigma(n)}\},$$
where $S_n$ denotes the permutation group in $n$ symbols. The
matrix $A $ is  \textbf{tropically singular} if this \emph{maximum
 is attained, at least, twice}. Otherwise, $A$ is \textbf{tropically regular}.

\m


The \textbf{projective tropical $n-1$--dimensional space}, denoted $\T\P^{n-1}$ is the quotient  $\R^{n}/\sim$, where $(\sucen an)\sim(\sucen bn)$ if and only if  $(\sucen an)=(\sucen {\lambda+b}n)$, for some $\lambda\in\R$. The  class of $(\sucen an)$ will be denoted $[\sucen an]$.
Let $\sucen Xn$ denote the coordinates on $\T\P^{n-1}$. 
For any point in $\T\P^{n-1}$, we can choose a representative $(\sucen a{n-1},0)$, in which case we say that we
\textbf{work in $X_n=0$}. This allows us to identify $\T\P^{n-1}$ with $\R^{n-1}$. We will make this identification throughout the paper.

From now on, coordinates of points will be written in columns.

\m

\textbf{Linear spaces in $\T\P^2$ and $\TP$.}
It is well--known that tropical lines, planes, hyperplanes are  piece--wise linear and piece--wise convex objects;
see \cite{Akian_HB,Baccelli,Butkovic,Cuninghame_New,Gathmann, Gaubert,Gunawardena,Itenberg,Mikhalkin_T,Richter,Viro_D}. Let us call the  pieces (linear convex sets)
\textbf{building blocks}. For instance, the building blocks of tropical lines  are unbounded rays and, sometimes, segments. And  the building blocks of tropical  planes are
unbounded quadrants. These building blocks have rational slopes, i.e., orthogonal vectors to them can be chosen to have integer coordinates.

\m
Fix $a_j\in \R$. The tropical
linear form
\begin{equation}
a_1\odot X_1\oplus a_2\odot X_2\oplus a_3\odot X_3=\max\{a_1+ X_1, a_2+ X_2, a_3+ X_3\}
\end{equation}
defines a tropical line  in $\T\P^2$, denoted  $\Pi_a$: it is the set of points $[x_1,x_2,x_3]^t$ in $\T\P^2$ where the
\textbf{maximum is attained twice, at least.} Notice that the maximum is attained three times at $[-a_1,-a_2,-a_3]^t$. This point
is called the \textbf{vertex}  of $\Pi_a$.
Working in $X_3=0$ (i.e., identifying  $\T\P^2$ with $\R^2$ in a certain way), the graphical representation of $\Pi_a$ is the union of three rays $\suce r$ meeting at point $(a_3-a_1,a_3-a_2)$. The ray $r_j$ points towards the negative $X_j$ direction, for $j=1,2$, and $r_3$ points towards the positive $X_1=X_2$ direction;   see right--hand--side of
figure \ref{figure_01}.

%
%


The tropical
linear form
\begin{equation}
a_1\odot X_1\oplus a_2\odot X_2\oplus a_3\odot X_3\oplus a_4\odot X_4=\max\{a_1+ X_1, a_2+ X_2, a_3+ X_3, a_4+ X_4\}
\end{equation}
defines a tropical plane  $\Pi_a$ in $\TP$:  $\Pi_a$ is the set of points $[x_1,x_2,x_3,x_4]^t$ in $\TP$ where the \textbf{maximum is attained twice,
at least.} Notice that the maximum is attained four times at  $[-a_1,-a_2,-a_3,-a_4]^t$. This special point, denoted $v^{\Pi_a}$, is called
the \textbf{vertex}  of $\Pi_a$.
Working in $X_4=0$, we get $\Pi_a$ as a subset of $\R^3$: it is the union of six  2--dimensional  quadrants; see figure \ref{figure_02}.

\begin{figure}[h]
 \centering
  \includegraphics[keepaspectratio,width=14cm]{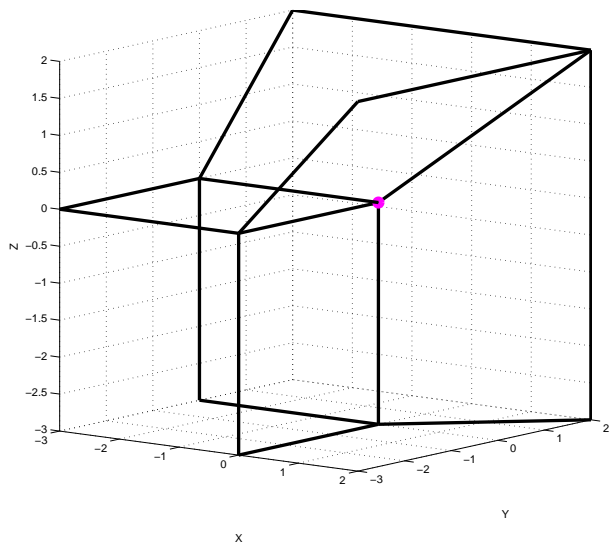}\\
  \caption{A tropical plane is the union of six closed quadrants.}
  \label{figure_02}
\end{figure}

If $a_4=0$ then, the quadrants of $\Pi_a$ are:
\begin{itemize}
\item $X_1=a_1$, $X_2\le a_2$, $X_3\le a_3$, denoted $Q_1$,\label{not:quadrants}
\item $X_1\le a_1$, $X_2=a_2$, $X_3\le a_3$, denoted $Q_2$,
\item $X_1\le a_1$, $X_2\le a_2$, $X_3=a_3$, denoted $Q_3$,
\item $a_1\le X_1=X_2$, $X_3\le X_1=X_2$, denoted $Q_{12}$,
\item $a_2\le X_2=X_3$, $X_1\le X_2=X_3$, denoted $Q_{23}$,
\item $a_3\le X_3=X_1$, $X_2\le X_3=X_1$, denoted $Q_{31}$.
\end{itemize}

\m
It is well--known that three tropically non--collinear points $p,q,r$  in $\TP$ determine a unique tropical plane, denoted  $L(p,q,r)$,  which passes through them; see
\cite{Richter}. The vertex  of $L(p,q,r)$ is computed from the coordinates of $p,q,r$ by the \emph{tropical Cramer's rule}; see \cite{Richter,Tabera_Pap}.
Inside $L(p,q,r)$, the points $p,q,r$ can be arranged as follows:
\begin{itemize}
\item all three in one quadrant, or
\item two in one quadrant, the third one in another quadrant, or
\item each one in a different quadrant (this is the generic case).
\end{itemize}

\begin{figure}[h]
 \centering
  \includegraphics[keepaspectratio,width=14cm]{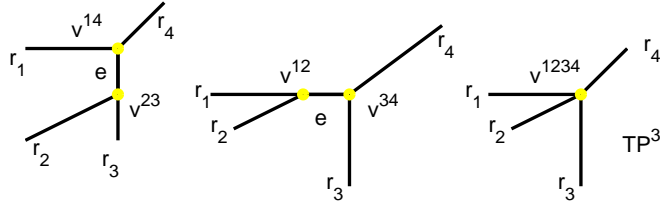}\\
  \caption{Graphical representation of some tropical lines in $\TP$: types  $[14,23]$ on the left, $[12,34]$ center,  tetrapod, on the right. These objects lie in $\R^3$; they are non--planar. The ray $r_4$ points towards the positive direction $X_1=X_2=X_3$.}
  \label{figure_03}
\end{figure}

We have already explained that a  
tropical line  in $\T\P^2$ is a \textbf{tripod}: it has one vertex and three rays. Now, a tropical line in $\TP$ is \textbf{not homeomorphic} to a  tripod (and this is a crucial difference between classical and tropical mathematics: lines in $\P^n$ and $\P^m$ are homeomorphic, if $n\neq m$). 
\label{dfn:lines} A  generic tropical line $L$ in $\TP$ has two vertices. Its building blocks are one edge $e$ (i.e., a segment of finite length) and four unbounded rays $r_1,r_2,r_3,r_4$. Ray $r_j$ points in the $j$--th  negative  coordinate direction, for $j=1,2,3$ and  $r_4$ points towards the positive direction of $X_1=X_2=X_3$; see  figure \ref{figure_03} (unfortunately, our planar graphical representation of tropical lines in $\TP$  cannot show angles properly).

\m

Any line $L$  in $\TP$ belongs to one of the following types:
\begin{equation}
[12,34],\qquad [13,24],\qquad [14,23],\qquad [1234].
\end{equation}
Let us explain further. For type $[i j, k l]$, the vertices of $L$ will be named $v^{i j}$ and $v^{kl}$ and the edge $e$  joins them.  Moreover,  $e$, $r_i$ and $r_j$ meet  at $v^{ij}$; same for  $e$, $r_k$, $r_l$ and $v^{kl}$.
 If the edge $e$   does not exist in a given  line $L$, then the two vertices of $L$ coincide and the type of $L$ is $[1234]$.   
In this case we say that $L$ is a \textbf{tetrapod}. \label{dfn:tetrapod}
Types of tropical lines in $\TP$ can be written in various ways: for example,  $[12,34]=[21,34]=[21,43]=[34,12]$.

\m
Let us see how do types arise.
It is well--known that two different points $p,q$  in $\TP$ determine a unique tropical line, denoted $L(p,q)$, which passes through them.\label{dfn:line_pq} Following \cite{Speyer_Sturm}, the type and vertices of $L(p,q)$ are computed  as follows.
For $1\le i<j\le4$, consider  the   $2\times 2$ tropical minors:
\begin{equation}\label{not:mij}
m_{ij}=\left|\begin{array}{cc} p_i&q_i\\p_j&q_j\\
\end{array}\right|_{\trop}\nega.
\end{equation} These minors satisfy the \textbf{tropical Pl\"{u}cker relation}, i.e., the following maximum is attained twice, at least:
\begin{equation}
\max\{m_{12}+m_{34}, m_{13}+m_{24}, m_{14}+m_{23}\}.
\end{equation}
Then
\begin{itemize}
\item type $[12,34]$ arises when $m_{12}+m_{34}<m_{13}+m_{24}=m_{14}+m_{23}$,
\item type $[13,24]$ arises when $m_{13}+m_{24}<m_{12}+m_{34}=m_{14}+m_{23}$,
\item type $[14,23]$ arises when $m_{14}+m_{23}<m_{12}+m_{34}=m_{13}+m_{24}$,
\item type $[1234]$ arises when $m_{12}+m_{34}=m_{13}+m_{24}=m_{14}+m_{23}$.
\end{itemize}

\m
Now, a natural question is to determine the line $L(p,q)$ (and its type), for two given points $p\neq q$. In order to do so, it is enough to determine the vertex  or vertices of $L(p,q)$.
A point $x$ belongs to $L(p,q)$ if and only if:
\begin{equation}
\rk\left[\begin{array}{ccc}
p_1&q_1&x_1\\p_2&q_2&x_2\\p_3&q_3&x_3\\p_4&q_4&x_4\\
\end{array}\right]_{\trop}\nega=2.
\end{equation}
This \textbf{tropical rank} condition  means that the value of each of the following $3\times 3$ tropical minors  is attained twice, at least (see \cite{Develin_Santos_Sturm} for tropical rank issues):
\begin{equation}\label{eqn:omit1}
\left|\begin{array}{ccc}
p_2&q_2&x_2\\p_3&q_3&x_3\\p_4&q_4&x_4\\
\end{array}\right|_{\trop}\nega =\max\{x_2+m_{34},x_3+m_{24},x_4+m_{23}\}
\end{equation}
\begin{equation}\label{eqn:omit2}
\left|\begin{array}{ccc}
p_1&q_1&x_1\\p_3&q_3&x_3\\p_4&q_4&x_4\\
\end{array}\right|_{\trop}\nega =\max\{x_1+m_{34},x_3+m_{14},x_4+m_{13}\}
\end{equation}
\begin{equation}\label{eqn:omit3}
\left|\begin{array}{ccc}
p_1&q_1&x_1\\p_2&q_2&x_2\\p_4&q_4&x_4\\
\end{array}\right|_{\trop}\nega =\max\{x_1+m_{24},x_2+m_{14},x_4+m_{12}\}
\end{equation}
\begin{equation}\label{eqn:omit4}
\left|\begin{array}{ccc}
p_1&q_1&x_1\\p_2&q_2&x_2\\p_3&q_3&x_3\\
\end{array}\right|_{\trop}\nega =\max\{x_1+m_{23},x_2+m_{13},x_3+m_{12}.\}
\end{equation}
Now, for any $u $  positive, big enough  real number, it is obvious that the points $x_j(u)$ below make the maxima attained, at least, twice, in expressions (\ref{eqn:omit1}), (\ref{eqn:omit2}), (\ref{eqn:omit3}), (\ref{eqn:omit4}),
respectively:
\begin{equation}\label{eqn:r12}
x_1(u)=[-u , -m_{34},-m_{24},-m_{23}]^t\quad x_2(u)=[ -m_{34},-u ,-m_{14},-m_{13}]^t,
\end{equation}
\begin{equation}\label{eqn:r34}
x_3(u)=[-m_{24},-m_{14},-u , -m_{12}]^t,\quad x_4(u)=[ -m_{23},-m_{13},-m_{12}, -u ]^t
\end{equation}
  And as $u $ goes from zero to $+\infty$, the point  $x_j(u)$  moves along  a ray $r_j$.

Say the type of $L(p,q)$ is $[12,34]$.
Then  a value of $u $ can be determined so that $x_1(u)=x_2(u)=v^{12}$ (resp. $x_3(u)=x_4(u)=v^{34}$), obtaining the following vertices for $L(p,q)$:
\begin{equation}\label{eqn:v^12}
v^{12}=[m_{13}-m_{23}-m_{34},-m_{34},-m_{24},-m_{23}]^t.
\end{equation}
\begin{equation}\label{eqn:v^34}
v^{34}=[-m_{24},-m_{14},m_{13}-m_{12}-m_{14},-m_{12}]^t.
\end{equation}

Say the type of $L(p,q)$ is $[13,24]$.
Similar calculations yield  the following vertices for $L(p,q)$:
\begin{equation}\label{eqn:v^13}
v^{13}=[m_{24},-m_{14}, -m_{24}-m_{14}+m_{34},-m_{12}]^t.
\end{equation}
\begin{equation}\label{eqn:v^24}
v^{24}=[-m_{23},-m_{13},-m_{12}, -m_{13}-m_{12}+m_{14}]^t.
\end{equation}

Say the type of $L(p,q)$ is $[1234]$.
Then   $x_1(u)=x_2(u)=x_3(u)=x_4(u)$  for some  $u $, providing the vertex
\begin{equation}
v^{1234}=[m_{13}+m_{14}-m_{34},m_{12},m_{13},m_{14}]^t.
\end{equation}
 Computations are similar for type $[14,23]$.

\m
Notice that the coordinates of  the vertices of $L(p,q)$ depend on the type of $L(p,q)$.
\m

\textbf{Lines as tropical algebraic varieties.}
An integer vector is \emph{primitive} if its coordinates are relatively prime.
It is well--known that a \emph{tropical algebraic variety} $V$ satisfies the \textbf{balance condition}  
(see \cite{Gathmann,Itenberg,Mikhalkin_W,Richter}) at
each point $p\in V$: this means that
\begin{equation}
w(p)_1+\cdots+w(p)_{s(p)}=0,
\end{equation}
where $s(p)\ge2$ and  $w(p)_1,\ldots,w(p)_{s(p)}$ are all the \emph{weighted primitive vectors} which are outward normal to the
different $s(p)$ building blocks of $V$ meeting at $p$.

\m

In  p. \pageref{dfn:lines}  we have seen that lines in $\T\P^2$ are \emph{not homeomorphic} to lines  in $\TP$. But, as tropical varieties, lines must
be all the same (i.e., they must be equivalent, in some way), regardless of the embedding dimension. The key concept to get such an equivalence relation is called  \textbf{modification}; see \cite{Mikhalkin_T} for details. 
We  will briefly explain modifications
only for lines.\label{dfn:modification}
    Suppose $L\subset \TP$ is a tropical line in which rays $r$ and $r'$ meet at vertex  $v\in L$ (and either an edge $e$ or two more rays meet also at $v$). The balance condition for $L$ holds at $v$. Roughly speaking, \textbf{a modification of $L$
at the point $v$} consists in contracting one ray (either $r$ or $r'$)
 down to the point  $v$ obtaining something, denoted $\overline{L}$, as a result. $\overline{L}$  must satisfy the balance condition, so at the same time of contracting one ray, a  straightening  of the direction of the second ray is necessary. The resulting object $\overline{L}$,
viewed in $\T\P^2$, satisfies the balance condition at every point, also at $v$. $\overline{L}$ is is a tropical line in $\T\P^2$, called  a modification of $L$ at $v$.
The inverse procedure, i.e., passing from $\overline{L}$  to $L$ is also called a modification;
see figure \ref{figure_01}, with $r=r_1$ and $r'=r_2$ or $r=r_2$ and $r'=r_1$.

\m
The \textbf{tropical distance} between two points  $p=[p_1,p_2,p_3,0]^t$ and $q=[q_1,q_2,q_3,0]^t$ is
\begin{equation}\label{dfn:trop_dist}
\tdist(p,q):=\max\{|p_i-q_i|\, :\, i=1,2,3\}.
\end{equation}
 Notice that if three coordinates in $p$ and $q$ coincide, then the tropical and Euclidean distances between $p$ and $q$ coincide.

\m
Let $n,m\in\N$ and $a_1,\ldots, a_n$ be  points in $\T\P^m$.
The \textbf{tropical span} of $a_1\ldots,a_n$ is
\begin{equation}\label{eqn:span}
\spann(a_1,\ldots,a_n):=\{(\lambda_1+ a_1)\oplus\cdots\oplus (\lambda_n+ a_n) \in \T\P^m: \lambda_1,\ldots,\lambda_n\in\R\},
\end{equation}
where maxima are computed coordinate--wise.
If $n=2,3$ or $4$, we speak of  \textbf{tropical segment}, \textbf{tropical triangle} or \textbf{tropical tetrahedron}.



 Assume $m=n-1$ and let us write the coordinates of the $a_j$ as the columns of a matrix $A$.
In order to view $\spann(A)$ inside $\R^{n-1}$, we must use the matrix $A_0=(\alpha_{ij})$, \label{dfn:alphas} where
\begin{equation}\label{eqn:alpha}
\alpha_{ij}=a_{ij}-a_{nj}.
\end{equation}
By (\ref{eqn:span}),   $x=[\sucen x{n-1},0]^t\in \spann(A)$ if and only if there exist $\sucen \lambda m\in\R$ such that
\begin{equation}
0=\max_{k\in[m]}\lambda_k,\quad
x_j=\max_{k\in[m]}\lambda_k+\alpha_{kj},\quad j\in[n-1].\label{eqn:max_otra}
\end{equation}

\m
\textbf{Kleene stars.} Consider $A\in \R^{n\times n}$.  By definition (see \cite{Butkovic_Libro,Sergeev,Butkovic_Sch_Ser,Sergeev_Sch_But}), $A$ is \emph{a Kleene star} if $A$ is null--diagonal and idempotent, tropically; in symbols: $\diag(A)=0$ and $A=A^2$. Notice that, if every diagonal entry of $A=(a_{ij})$ vanishes, then  $A\le A^2$, because   for each $i,j\in[n]$, we have $(A^2)_{ij}=\max_{k\in[n]} a_{ik}+a_{kj}\ge a_{ij}$. Therefore, being a Kleene star is characterized by the following  $n$  equalities and ${n\choose 2}+{n\choose 3}$
linear inequalities:
\begin{equation}\label{eqn:Kleene}
a_{ii}=0,\quad a_{ik}+a_{kj}\le a_{ij}, \quad i,j,k\in[n], \quad \card\{i,j,k\}\ge2.
\end{equation}
%

\section{Convex maximal tropical tetrahedra in $\TP$}\label{sec:tetra}

Assume $A$ is an order $4$ matrix.
The objective of this paper is to find all the possible combinatorial types of the tropical tetrahedron $\spann(A)$, when $\spann(A)$ is convex and maximal (maximality to be defined in p. \pageref{dfn:maximal}). By \cite{Puente_convex,Sergeev},  $\spann(A)$ is convex if and only if $A$ is a Kleene star, and then $\spann(A)$ is determined by the following 12 inequalities:
\begin{equation}\label{eqn:eqns}
 \begin{tabular}{c|c}
 $a_{14}\le X_1\le -a_{41}$&$a_{12}\le X_1-X_2\le -a_{21}$\\
 $a_{24}\le X_2\le -a_{42}$&$a_{23}\le X_2-X_3\le -a_{32}$\\
 $a_{34}\le X_3\le -a_{43}$&$a_{31}\le X_3-X_1\le -a_{13}$.
\end{tabular}
\end{equation}

From now on, we will assume that $A$ is an order 4 Kleene star. For $i,j\in[4]$, $i\neq j$, let $L_{ij}:=L(\col(A,i),\col (A,j))$.

\subsection{The type of the tropical line $L_{ij}$}

Here is an application of the balance condition in $\TP$ identified with $\R^3$. Let $e_1,e_2,e_3$ denote the canonical vectors, \label{dnf:canonical} $e_{ij}:=e_i+e_j$, for $i\neq j$ and  $e_{123}:=e_1+e_2+e_3$.
Suppose $L$ is a tropical line, not a tetrapod. Then the direction of the edge of $L$ is $e_{ij}$ if and only if the type of $L$ is $[ij,k4]$,  with $\{i,j,k\}=[3]$.

\begin{thm}\label{thm:types} Assume $A$ is an order $4$ Kleene star.
Let   $\{i,j,k,l\}=[4]$ with  $i<j$. Then
 the \textbf{type} of the tropical line $L_{ij}$ is either $[ik,jl]$ or $[il,jk]$ or else $L_{ij}$ is a tetrapod  (easy to remember: \textbf{$i$ and $j$ are separated by the comma}, unless  $L_{ij}$ has just one vertex).
\end{thm}

\begin{proof}
Without lost of generality,  assume that $i=1$, $j=2$, so that $\{k,l\}=\{3,4\}$. Write  $p=\col(A,1)$ and $q=\col(A,2)$ and compute the tropical minors $m_{ij}$'s as in expression (\ref{not:mij}). Write $M=|A(34;12)|_{\trop}$. Then
$$m_{12}+m_{34}=M,\ m_{13}+m_{24}=a_{32}+a_{41},\ m_{14}+m_{23}=a_{31}+a_{42}.$$ The  value $M$ is attained at the main (resp. secondary) (resp. both) diagonal(s) if and only if  the type of line $L_{12}$ is $[13,24]$ (resp. $[14,23]$) (resp. $[1234]$). This proves the  statement.
\end{proof}


\subsection{Generators and extremals}

From now on, we assume that $A$ is an order 4 Kleene star such that \textbf{the columns of $A$ represent four tropically non--coplanar points in $\TP$}: this is our \textbf{hypothesis 1}. 

\label{rem:hypothesis}
In \cite{Joswig_K},  a tropical tetrahedron is defined as the tropical  tropical span of \textbf{four points which are not contained in the boundary of a tropical halfspace}. If four points are not tropically coplanar, then they are not contained in the boundary of a tropical halfspace, but the converse is not true.  Therefore, our hypothesis 1 is  more restrictive (and more natural, in our opinion) than the hypothesis in \cite{Joswig_K}. In particular, if  (within a smaller set of matrices) we find six combinatorial classes (see below, p. \pageref{class:first}), then there should be at least six classes in \cite{Joswig_K}, but only five classes are shown there.

\m
Our aim  is to study $\spann(A)$ as a convex body in 3--dimensional space. It  turns out that $\spann(A)$ is
not regular, i.e., its facets are irregular polygons. However, the facets of $\spann(A)$ are nice enough, because they are contained in classical planes of equations
$X_i=\cnst$, $X_j-X_k=\cnst$, $i,j,k\in[3]$.
In other words, the edges of $\spann(A)$ have  directions   $e_i$, $e_{jk}$, or $e_{123}$.  Such polyhedra are called  \textbf{alcoved polyhedra};
see \cite{Joswig_K,Lam_Postnikov,Lam_PostnikovII, Werner_Yu}.

\m
We work in $\TP$, identified with $\R^3$, using the matrix $A_0=(\alpha_{ij})$ defined in (\ref{dfn:alphas}).
The points represented by the columns of $A_0$ will be called \textbf{generators}. Taking generators three by three,
they yield four additional points (the four vertices of the four corresponding tropical planes)
and taking generators two by two, they yield, at most,  twelve more points (the vertices of  six  tropical lines). These new points will be called \textbf{extremal generated points}. Generators and extremal generated points  are called \textbf{extremals} of $\spann(A)$.
In the generic  case, we get a total of $20=4+4+12$ \emph{different extremals}  and then we say that
$A$ is \textbf{maximal in extremals}. Notice that $20={{2n-2}\choose{n-1}}$, for $n=4$, which agrees with \cite{Develin_Sturm,Gelfand_G,Joswig_K}.  
The number and computation of extremals have been  studied in \cite{Allamigeon_GG,Allamigeon_GK}, in a more general setting (where classical convexity is not assumed).

\m

The following \textbf{color code for figures} will be used: blue for generators,  magenta for vertices of tropical planes and  yellow for vertices of tropical lines. Two adjacent red segments should be glued together, after cutting and folding. Dashed segments must be mountain--folded, dotted  segments must be valley--folded.

\subsection{Some tropical triangles in $\TP$, their matrices and  vertex  configurations}

Now, we must sidetrack to  discuss what
some tropical triangles in $\TP$ are like.\label{par:trian} The columns of $A_0$, taken three by three, determine four tropical triangles in $\TP$. Tropical triangles in $\TP$ can be easily understood.
Planar tropical triangles  have been studied in \cite{Ansola_tri,Develin_Sturm,Joswig_K,Puente_lin}. In general, tropical triangles are compact but not convex.

Let $B$ be a $4\times 3$ real  matrix obtained by deleting one column in $A$. We want to describe the tropical triangle $\spann(B)$. The points represented by the columns of $B_0$  will be called the \textbf{generators} of $\spann(B)$.
They determine one tropical plane $\Pi^B$ and three tropical   lines: $L_{12}$, $L_{23}$ and $L_{31}$, where $L_{ij}$ denotes $L(\col(B,i),\col (B,j))$.   Thus, \textbf{additional
extremals}  arise: the vertex of the plane $\Pi^B$  and the vertices of the  lines.  In the generic  case,
we get a total of $10=3+1+2\times 3$ different extremals in $\spann(B)$ and we say that $B$ is \textbf{maximal in extremals}.

The \textbf{f--vector} of $\spann(B)$ is $(v,e,f)$, where $v$ (resp. $e$) (resp. $f$) is the number of extremals (resp. edges)
(facets) of $\spann(B)$. The \emph{Euler characteristic} of $\spann(B)$ is $v-e+f=1$.  The \textbf{polygon--vector} of $\spann(B)$
is $(f_3,f_4,f_5,f_6)$, where $f_m\ge0$ is the number of convex $m$--gons occurring as facets and  $f=f_3+f_4+f_5+f_6$. \label{dfn:polygonvector}

\m
We will only consider the maximal case, i.e., $v=10$. 
Inside $\Pi^B$, the generators of $\spann(B)$ can sit as follows:
\begin{itemize}
\item all three in one quadrant (therefore, the  triangle $\spann(B)$ is planar), or \label{it:in1}
\item two in one quadrant, the third one in another quadrant, or\label{it:in2}
\item each generator in a different quadrant (this is the generic case).\label{it:in3}
\end{itemize}

\begin{figure}[h]
 \centering
  \includegraphics[keepaspectratio,width=14cm]{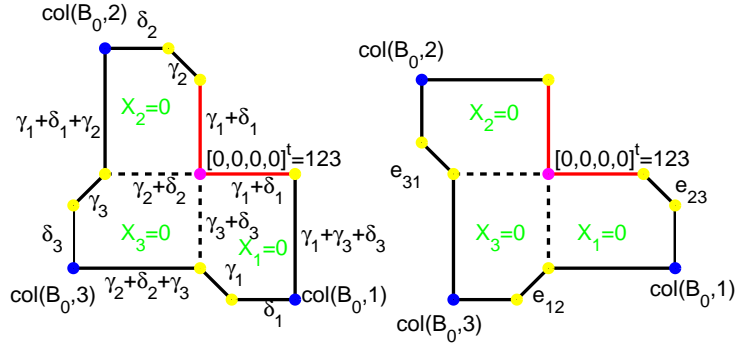}\\
  \caption{Cut--and--fold models for a tropical triangle $\spann(B)$ made up of three pentagons. Point $v^{\Pi^B}$ is of type (5.5.5): left and right models. In the figure, the coordinates of $v^{\Pi^B}$ are $[0,0,0,0]^t$.}
  \label{figure_04}
\end{figure}

We need only study the generic case (i.e., $f=3$), since the second case (i.e., $f=2$) is a degeneration of it,  and the first case (i.e., $f=1$) has already been studied in \cite{Ansola_tri,Develin_Sturm,Joswig_K,Puente_lin}.
In the generic case,  $\spann(B)$ is the union of  three classical convex $m$--gons, one $m$--gon contained in  each quadrant,
 with $m=3,4,5,6$.  
 How are the ten extremals of $\spann(B)$ distributed among the quadrants
 of $\Pi^B$? By genericity,  each  generator lies on a different  quadrant and  only three quadrants are involved. Also, the vertex  of $\Pi^B$ is  common to
  all three quadrants. And,
 for each two quadrants (out of three), one additional extremal point lies on their intersection. This leaves out three extremals.  How are these three   distributed among the three quadrants involved? They cannot  lie all on just one  quadrant, because this would yield more than six extremal
 points (generators or additional) on one quadrant, but our polygons have six vertices, at most.
So, the three remanent extremals can be arranged as follows:
 \begin{itemize}
 \item  one in each quadrant, or \label{it:uno_en_cada}
 \item  two in one quadrant and one in another quadrant. \label{it:en_dos}
 \end{itemize}

\begin{figure}[h]
 \centering
  \includegraphics[keepaspectratio,width=14cm]{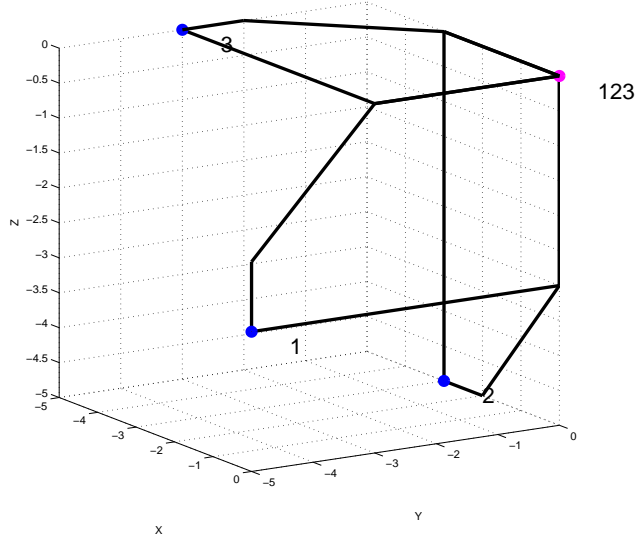}\\
  \caption{3--dimensional  model for matrix (\ref{eqn:3penta_left})  with $\gamma_j=2$, $\delta_j=1$, $j=1,2,3$ left.}
  \label{figure_05}
\end{figure}

\begin{figure}[h]
 \centering
  \includegraphics[keepaspectratio,width=14cm]{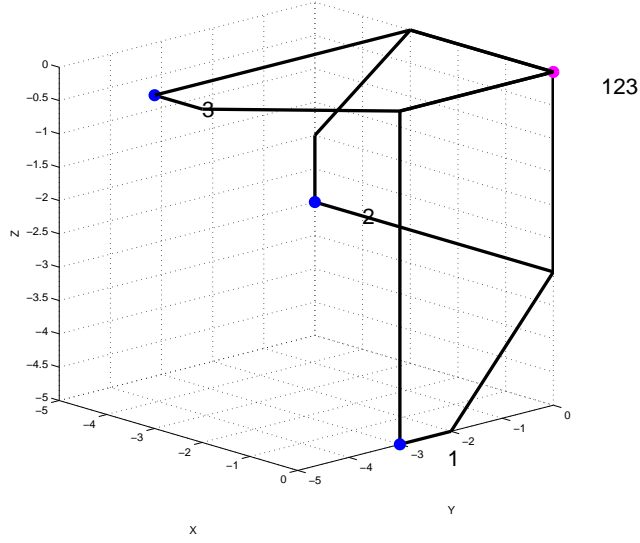}\\
  \caption{3--dimensional  model for matrix (\ref{eqn:3penta_right})  with $\gamma_j=2$, $\delta_j=1$,  $j=1,2,3$ right.}
  \label{figure_06}
\end{figure}

 In the first case, the tropical triangle $\spann(B)$ is the union of \textbf{three pentagons}, i.e., the polygon--vector of $\spann(B)$ is $(0,0,3,0)$. This situation is also expressed
 by saying that the \textbf{configuration at the point} $v^{\Pi^B}$  is  $(5.5.5)$. There is a left version and a right version of configuration $(5.5.5)$; see  cut--and--fold  models in figure
 \ref{figure_04}, and the corresponding folded models in figures  \ref{figure_05} and \ref{figure_06}.

 In this case, \textbf{what is $B_0$ like?} Let us  assume that $v^{\Pi^B}=[0,0,0,0]^t$ and the deleted column in $A$ is the fourth one. Then

 \begin{equation}\label{eqn:3penta_left}
 B_0=\left[
 \begin{array}{rrr}
 0&-\gamma_2-\delta_2&-\gamma_3-\gamma_2-\delta_2\\
 -\gamma_1-\gamma_3-\delta_3&0&-\gamma_3-\delta_3\\
 -\gamma_1-\delta_1&-\gamma_2-\gamma_1-\delta_1&0\\
 0&0&0
 \end{array}
 \right] \text{left\ model},
 \end{equation}
and
\begin{equation}\label{eqn:3penta_right}
 B_0=\left[
 \begin{array}{rrr}
 0&-\gamma_2-\gamma_3-\delta_3&-\gamma_3-\delta_3\\
 -\gamma_1-\delta_1&0&-\gamma_3-\gamma_1-\delta_1\\
 -\gamma_1-\gamma_2-\delta_2&-\gamma_2-\delta_2&0\\
 0&0&0
 \end{array}
 \right]\text{right\ model},
 \end{equation}
 for some parameters $\suce \gamma,\suce \delta>0$. 

\begin{figure}[h]
 \centering
  \includegraphics[keepaspectratio,width=14cm]{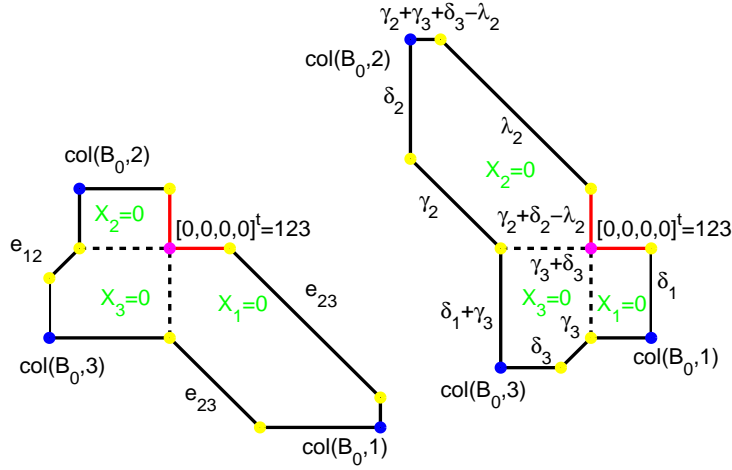}\\
  \caption{Cut--and--fold models for a tropical triangle $\spann(B)$ made up of a quadrangle, a pentagon and a hexagon. The configuration of  $v^{\Pi^B}$ is   $(6.4.5)$ for the left figure, and $(4.6.5)$ for the right figure.}
  \label{figure_07}
\end{figure}

\begin{figure}[h]
 \centering
  \includegraphics[keepaspectratio,width=14cm]{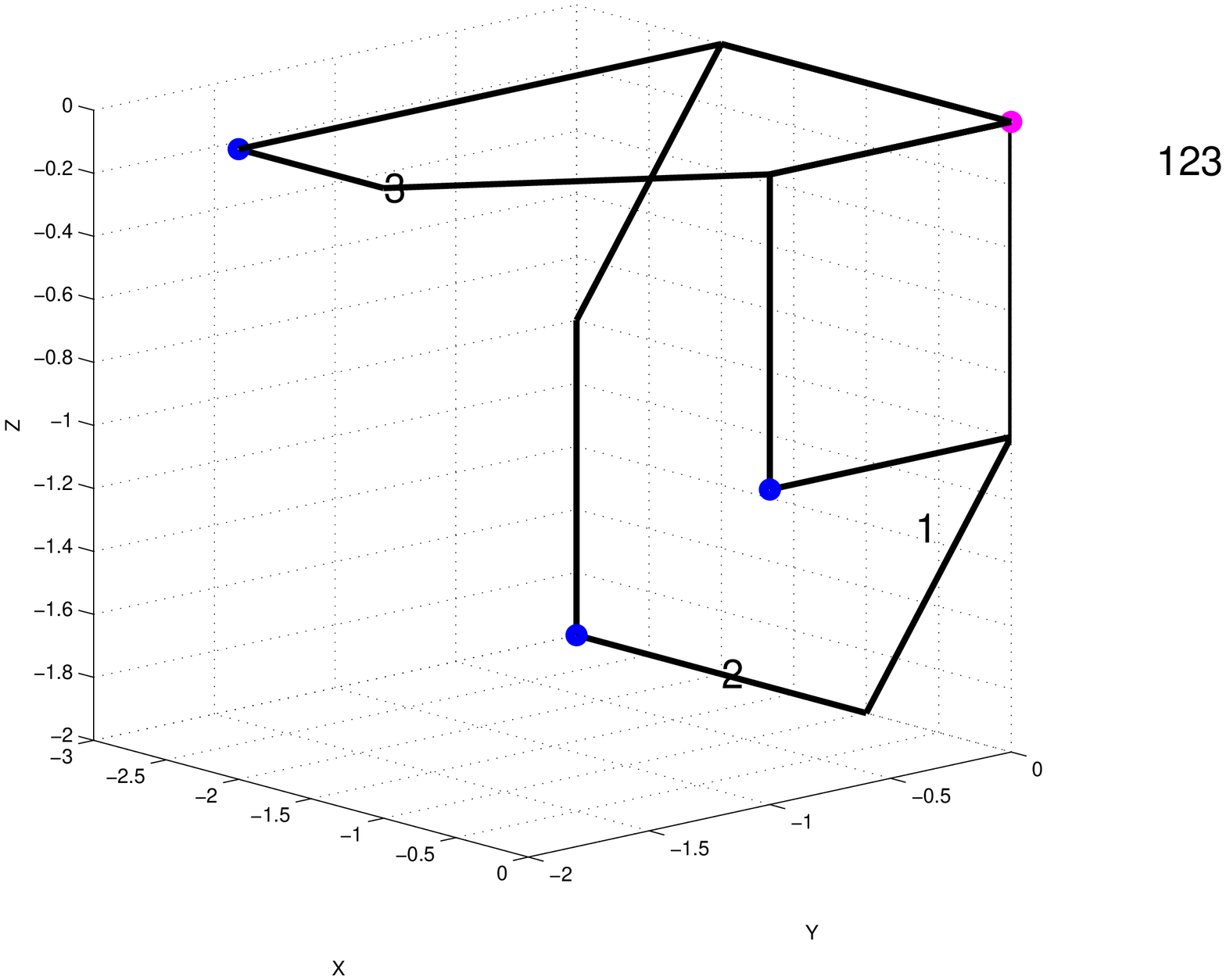}\\
  \caption{3--dimensional  model for matrix (\ref{eqn:456_right})  with $\delta_1=\gamma_2=\delta_2=\lambda_2=\gamma_3=\delta_3=1$.
  The configuration is $(4.6.5)$.}
  \label{figure_08}
\end{figure}

\m
In the second case, the tropical triangle $\spann(B)$ is the union of \textbf{a quadrangle, a pentagon and a hexagon}, i.e.,
 the  polygon--vector of $\spann(B)$ is $(0,1,1,1)$.  We say that the \textbf{configuration at the point} $v^{\Pi^B}$ is $(4.5.6)$ if the quadrangle is contained in the plane $X_1=\cnst$, the pentagon  in  $X_2=\cnst$, and the hexagon  in  $X_3=\cnst$. Similarly, we have configurations   $(p.q.r)$ for any $\{p,q,r\}=\{4,5,6\}$.    See figures \ref{figure_07}
 and \ref{figure_08} for a picture of some of these configurations.

\textbf{What is $B_0$ like?} Let us  assume $v^{\Pi^B}=[0,0,0,0]^t$ and the deleted column in $A$ is the last one. After a change of variables  we have
\begin{equation}\label{eqn:456_right}
 B_0=\left[
 \begin{array}{rrr}
 0&-\gamma_2-\gamma_3-\delta_3&-\gamma_3-\delta_3\\
 -\delta_1&0&-\delta_1-\gamma_3\\
 -\gamma_2-\delta_2+\lambda_2&-\gamma_2-\delta_2&0\\
 0&0&0
 \end{array}
 \right],\text{configuration $(4.6.5)$},
 \end{equation}
 for some parameters $\delta_1, \gamma_2,\delta_2,\lambda_2,\gamma_3,\delta_3>0$, with $\gamma_2+\delta_2-\lambda_2>0$, and
\begin{equation}\label{eqn:456_left}
B_0=\left[
 \begin{array}{rrr}
 0&-\delta_2&-\delta_2-\gamma_3\\
 -\gamma_1-\gamma_3-\delta_3&0&-\gamma_3-\delta_3\\
 -\gamma_1-\delta_1&-\gamma_1-\delta_1+\lambda_1&0\\
 0&0&0\\
\end{array}
 \right], \text{configuration $(6.4.5)$},
 \end{equation}
for some parameters $\gamma_1,\delta_1,\lambda_1,\delta_2,\gamma_3,\delta_3>0$, with $\gamma_1+\delta_1-\lambda_1>0$.
We get similar expressions for $B_0$, for other configurations $(p.q.r)$.

In summary,  if $B$ is generic and maximal in extremals (i.e., $f=3$  and $v=10$), then the possible configurations at the vertex of  $\spann(B)$ are: \begin{enumerate}\label{par:summary}
\item $(5.5.5)$ left,
\item $(5.5.5)$ right, \nopagebreak
\item $(p.q.r)$, with $\{p,q,r\}=\{4,5,6\}$.
\end{enumerate}

%
%
%
%
%
%
%
%
\m
 In the previous examples, the generators of $\spann(B)$ lie on  $Q_1$, $Q_2$ and $Q_3$, introduced
 in p. \pageref{not:quadrants}. These quadrants are orthogonal to each other, and so we say that the  \textbf{angle--vector}   at the vertex of  $\spann(B)$ is $\langle 90,90,90\rangle$, in degrees.
But if some generator lies on quadrant $Q_{ij}$, other angle--vectors will occur.
Indeed, the intersection of two quadrants can be parallel to  vector
 $e_j$, for $j=1,2,3$, or to  the vector $e_{123}$.   Thus, by elementary geometry, the  following angles occur:
 \begin{equation*}
 \overline{\alpha}:=\angle(e_j,e_{123})=\arccos(\sqrt{1/3})\simeq 54^\circ 44'
\end{equation*} and its supplementary
\begin{equation}\label{dfn:alpha}
 \alpha\simeq125^\circ  16'.
 \end{equation}
 So, the angle--vector $\langle90,\alpha,\alpha\rangle$ is also possible at the vertex  of a tropical triangle $\spann(B)$.


\m

 \label{dfn:angle--vector} Above, we have  considered the \textbf{angle--vector at the vertex of a tropical triangle}. We can also speak of  the \textbf{angle--vector of a polygon}. All the polygons that we will encounter below are planar tropical triangles.
 It is well--known that some tropical planar triangles are classical hexagons; see \cite{Ansola_tri}. The edges of these hexagons have  directions $e_1$, $e_2$
 and $e_{12}$,  the angles occurring there being
 \begin{equation}
 \angle(e_j,e_{12})=\arccos(\sqrt{1/2})=45^\circ ,\quad 180-45=135^\circ ,\quad 90^\circ ,
 \end{equation}
 whence
 $\langle 90,135,135,90,135,135\rangle$ is the angle--vector of such a hexagon; see figure \ref{figure_09}, upper left.  Pentagons, quadrangles
 or triangles in the same figure are obtained from the given  hexagon, when one, two or three edges have collapsed. The corresponding angle--vectors are thus
 easily deduced.
 Later on, we will also encounter 
 the following angles 
 \begin{equation*}\label{dfn:beta}
\overline{\beta}:=\angle(e_{jk},e_{123})=\arccos(\sqrt{2/3})\simeq 35^\circ 16'
\end{equation*} and its supplementary
\begin{equation}
 \beta\simeq 144^\circ  44'.
 \end{equation}\label{dfn:beta}

 \begin{figure}[h]
 \centering
  \includegraphics[keepaspectratio,width=14cm]{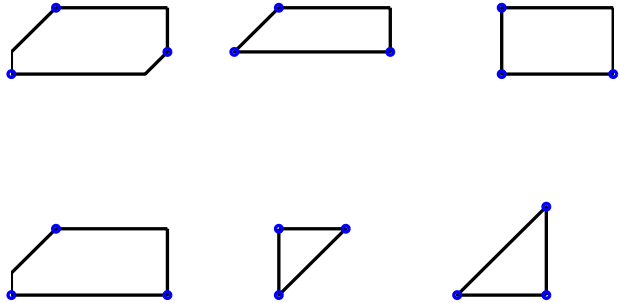}\\
  \caption{The planar case: some tropical triangles in $\T\P^2$ defined by tropically  idempotent   $3\times 3$ matrices. A hexagon is found in the upper left corner.
  The rest of the figures are obtained by letting some edges in such a hexagon collapse. Generators are marked in blue.}
  \label{figure_09}
\end{figure}

%
\subsection{Maximality and near--miss Johnson solids}

Let us return to discuss convex tropical tetrahedra. Since each tropical triangle in $\TP$ is made
up of, at most, 3  $m$--gons, for $m=3,4,5,6$, then  $\spann(A)$ may have up to
$4\times 3=12$ facets (this agrees with proposition 5 in \cite{Joswig_K}). If this is the case, we will say that  $A$ is \textbf{maximal in facets}. If $A$ is maximal in extremals
 (i.e., there are  20 such), then $A$ is maximal in facets, because each facet contains, at most, six extremals. The converse is not true.
 In the sequel, we will just say that $A$ is \textbf{maximal}, meaning maximal in extremals
  and facets. \label{dfn:maximal}
 By \cite{Joswig_K}, if $A$ is maximal, then $\spann(A)$  is  \textbf{simple} or \textbf{trivalent},
i.e., three facets are concurrent at each extremal point.

\m
By Euler's formula, with 20 vertices and 12 facets, $\spann(A)$ must have 30 edges. But  $(20,30,12)$
 is precisely the f--vector  of
a \emph{regular dodecahedron} $\cal D$ (the polygon--vector of $\cal D$ is, obviously, $(0,0,12,0)$). $\cal D$ is one of the famous  Platonic solids (or regular solids).  A \emph{Johnson solid} is a (less famous)  convex polyhedron, each facet of which is a regular polygon (like Platonic solids) but it is not uniform, i.e., it is not vertex--transitive (unlike Platonic solids). Since 1969, it has been known that there are exactly 92 classes of Johnson solids.   A convex polyhedron, each facet of which is near--regular is called a \textbf{near--miss Johnson solid.} This is a wide generalization of Johnson solids.
Let us visualize one near--miss Johnson solid. Take a regular dodecahedron $\cal D$ and choose four equidistant vertices in $\cal D$. Now allow each  chosen vertex to migrate to a neighboring facet. We obtain a new solid $\cal D'$, having  the same f--vector and  polygon--vector  $(0,4,4,4)$. $\cal D'$ is a near--miss Johnson solid. In lemma \ref{lem:circulant} we have two examples having the same combinatorial type  as  $\cal D'$. 
By tropicality and maximality, we will only deal with convex solids having $(20,30,12)$ as f--vector  and polygon--vector  $(0,f_4,f_5,f_6)$, with $12=f_4+f_5+f_6$.

\subsection{Oddly generated extremals}

Fix an order 4 Kleene star $A$. We must first name, then compute the extremals in $\spann(A)$.
Here are some notations:
\begin{itemize}
 \item Extremals of $\spann(A)$ will be underlined. In particular,
\underline{1}, \underline{2}, \underline{3}, \underline{4} are the generators. It is just an abbreviation for $\col(A_0,j)$, $j\in[4]$.

 \item  The vertex  of the tropical plane $L(\underline{i},\underline{j},\underline{k})$  is denoted $\underline{ijk}$. Here the  order of appearance if $i,j,k$ is irrelevant.

\item Let $\{i,j,k,l\}=[4]$.We say that the point \underline{$l$} is
    \textbf{opposite} to the point $\underline{ijk}$ inside $\spann(A)$. \label{dfn:opposite} 

\item We say that the point $\underline{ijk}$ is 3--generated. The extremals $\underline{i}$ and $\underline{ijk}$ are \textbf{generated by an odd number of points}. 
\end{itemize}

\begin{lem}\label{lem:transpose}
If $A$ is an order 4 Kleene star, then the columns of $-A^t$ represent the points \underline{234}, \underline{134}, \underline{124}, \underline{123}.
\end{lem}
\begin{proof}
The  coordinates of the points $\underline{ijk}$  are given by the tropical Cramer's rule; see \cite{Richter,Tabera_Pap}. This means that  the points $\underline{ijk}$  are given by the columns of the matrix $-\widehat{A}^t$. Here  $\widehat{A}=(b_{ij})$ is the \emph{tropical adjoint of $A$}, where $b_{ij}$ equals  the tropical determinant of the  minor obtained by omitting the $j$--th row and $i$--column in $A$. An easy computation shows that $\widehat{A}=A$, for a Kleene star $A$.
\end{proof}

Let $s:\R^{n-1}\rightarrow\R^{n-1}$ be the antipodal map, defined by $s(a)=-a$.
The following lemma is easy to prove.

\begin{lem}\ref{lem:transpose}
Let $A\in \R^{n\times n}$. The matrix $A$ is a Kleene star if and only if  $A^t$ is. In such a case, $\spann(A^t)=s(\spann(A))$. In particular, if $A$ is a Kleene star, then $A$ is symmetric if and only if $\spann(A)$ is symmetric with respect to the origin.\qed
\end{lem}

Set $n=4$.\label{rem:antipodal}
The map $s$ takes the extremal $\underline{i}$ in $\spann(A)$ onto $\underline{jkl}$ in $\spann(A^t)$.
Thus,  if $A$ is a maximal Kleene star,  the possible configurations at any oddly generated extremal point are summarized in  p. \pageref{par:summary}. For instance,
the configuration at  point \underline{4} is $(5.5.5)$  if three pentagons meet at \underline{4}. There are two possibilities, left and right, which are shown (unfolded) in figure \ref{figure_10}. At point \underline{4} we can also have configuration $(p.q.r)$, with $\{p,q,r\}=\{4,5,6\}$.

\begin{figure}[h]
 \centering
  \includegraphics[keepaspectratio,width=14cm]{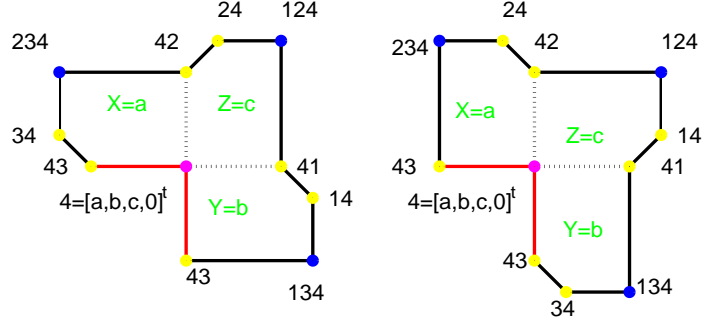}\\
  \caption{The configuration at point \underline{4} is $(5.5.5)$, left or right. Here, the coordinates of point \underline{4} are $[a,b,c,0]^t$.}
  \label{figure_10}
\end{figure}

\subsection{2--generated extremals, adiff and tropical distance}


%

Given two points $p\neq q$ in $\TP$, denote by $v,w$ the vertices of the tropical line $L(p,q)$.   
Pursuing maximality, we want to know \emph{when  $p,q,v,w$, are all different,} i.e., $\card\{p,q,v,w\}=4$.  

\begin{lem}\label{lem:dos_puntos}
Suppose $p,q$ are different points in $\TP$ and $v,w$  are the vertices (perhaps, $v=w$) of the tropical line $L(p,q)$. Then
$\card\{p,q,v,w\}\le3$ if and only if  then there exists a tropical change of variables that gives
$p=[0,0,a,b]^t$ and $q=[0,0,0,0]^t$,   for some  real numbers  $a,b$.
\end{lem}

\begin{proof}
Suppose that $p=[0,0,a,b]^t$ and $q=[0,0,0,0]^t$.
Write $M=a\oplus b$ and $a^+=a\oplus0$.  Recall here the notation $m_{ij}$ introduced in p.         \pageref{not:mij}. By
straightforward computations,  $$m_{12}+m_{34}=M,\quad m_{13}+m_{23}=m_{14}+m_{23}=a^++b^+.$$ Clearly, $M\le a^++b^+$ and $L(p,q)$ is a tetrapod if and only if $M=a^++b^+$ or, equivalently, $ab\le 0$. In this case, $\card\{p,q,v,w\}\le3$. Otherwise,  $ab>0$ and the type of $L(p,q)$ is $[12,34]$, with
$$v^{12}=[-M,-M,-b^+,-a^+]^t, \quad v^{34}=[-b^+-a^+,-b^+-a^+,-b^+,-a^+]^t.$$
 If we simplify the former expressions,
in each particular case, we obtain
\begin{enumerate}
\item if $a=b$, then the vertices of $L(p,q)$ are $p$ and $q$, so that $\card\{p,q,v,w\}=2$,
\item if $a\neq b$, then the vertices of $L(p,q)$ are either $p$ or  $q$, but not both, and one more point, so that $\card\{p,q,v,w\}=3$.
\end{enumerate}


Suppose now that $p=[a,b,c,0]^t$ and $q=[0,0,0,0]^t$, with non--zero  $a,b,c$ and not $a=b=c$. The tropical line $L(p,q)$ has four rays, one in each negative coordinate direction $e_j$, $j\in[3]$, and one in the positive direction $e_{123}$. Moreover,   at each point of $L(p,q)$ the balance condition holds. Since   $a,b,c$ are non--zero and they are different, we cannot  go from
$q$ to $p$ along $L(p,q)$ running through \textbf{only two} classical segments. Therefore,  $L(p,q)$ has a bounded edge $e$ (having
direction $e_{kl}$, for some different $k,l\in[4]$) and $\card\{p,q,v,w\}=4$.

%
\end{proof}

Here are
more notations for a given order 4 matrix $A$. Choose $i\neq j$ in $[4]$; the vertices of the  tropical line $L_{ij}$ are  $\underline{ij}$ and  $\underline{ji}$ named so that $\underline{ij}$ 
     is the \textbf{closest to $\underline{i}$}, with respect to  tropical distance. 
     Of course, $\underline{ji}=\underline{ij}$ if and only if  $L_{ij}$ is a tetrapod. We say that the extremals $\underline{ij}$ and $\underline{ji}$ are \textbf{2--generated}.

\m
Now we introduce \textbf{adiffs}, which provide the tropical distance between some pairs of extremals.
Let $i,j,k,l\in[n]$ with $k<l$ and $i<j$. By $A(kl;ij)$ we denote the $2\times 2$ minor $\left[\begin{array}{cc}
a_{ki}&a_{kj}\\
a_{li}&a_{lj}\\
\end{array}\right]$.
Write $\adiff_A (kl;ij)$ to denote
\begin{equation}
|a_{ki}+a_{lj}- a_{kj}-a_{li}|,
\end{equation}
 i.e., the \textbf{absolute value of the difference} of the items in the maximum below
 \begin{equation}
 |A(kl;ij)|_{\trop}=\max\{a_{ki}+a_{lj}, a_{kj}+a_{li}\}.
 \end{equation}
We extend  the notation as follows: $\adiff_A (lk;ij)=\adiff_A (kl;ji)=\adiff_A (lk;ji)$ are all equal to the already defined $\adiff_A (kl;ij)$, with $k<l$ and $i<j$.

\m
The following properties are easy to check, for $i,j,k,l\in[n]$:
\begin{enumerate}
\item $\alpha_{ii}-\alpha_{ij}=\adiff_A(in;ij)$,
\item $\adiff_A(ij;kl)=\adiff_{A_0}(ij;kl)$.
\end{enumerate}

We simply write $\dist(p,q)$, when the Euclidean and tropical distances between points $p$ and $q$ coincide.

\begin{thm}\label{thm:adiffs} Assume hypothesis 1 for an order 4 matrix $A$.
Let   $\{i,j,k,l\}=[4]$ with  $i<j$.
If the type of the tropical line $L_{ij}$ is $[ik,jl]$, then $$\dist (\underline{i}, \underline{ij})=
\adiff_A (jl;ij),\quad \dist (\underline{j}, \underline{ji})=
\adiff_A (ik;ij),$$
$$\tdist (\underline{ij}, \underline{ji})=\adiff_A (kl;ij).$$
\end{thm}

\begin{proof} 
Without lost of generality,  assume that $i=1$, $j=2$, so that $\{k,l\}=\{3,4\}$. Recall that the coordinates of the vertices of  $L_{12}$  depend on the type of $L_{12}$.

Say the type of $L_{12}$ is $[13,24]$; then $k=3$, $l=4$.  Using formulas (\ref{eqn:v^13}) and (\ref{eqn:v^24}), the vertices of line $L_{12}$ are
\begin{equation}\label{eqn:v^13,v^24}
v^{13}=\colur {0}{a_{41}-a_{42}}{a_{31}}{a_{41}}=\colur{-a_{41}}{-a_{42}}{a_{31}-a_{41}}{0},\quad v^{24}=\colur{a_{32}-a_{31}}{0} {a_{32}}{a_{42}}=\colur{a_{32}-a_{31}-a_{42}}{-a_{42}} {a_{32}-a_{42}}{0}.
\end{equation}

 The generators \underline{1} and \underline{2} have coordinates
$$\colu {0}{a_{21}}{a_{31}}{a_{41}},\quad \colu {a_{12}}{0}{a_{32}}{a_{42}},$$
respectively. Notice that three coordinates of $v^{13}$ and \underline{1} coincide and only the second one  is different. Therefore, the tropical distance and Euclidean distance between these two points coincide, being $a_{41}-a_{42}-a_{21}=|a_{41}-a_{42}-a_{21}|=\adiff_A (24;12)$, by inequalities (\ref{eqn:Kleene}). Moreover, we can check that the tropical distance between $v^{13}$ and \underline{2} is no less than $\adiff_A (24;12)$, whence $v^{13}=\underline{12}$ and $v^{24}=\underline{21}$.
Also, comparing $v^{24}$ and \underline{2}, only the first coordinate is different. Therefore the tropical distance and Euclidean distance between these two points coincide, being $a_{32}-a_{31}-a_{12}=|a_{32}-a_{31}-a_{12}|=\adiff_A (13;12)$. The tropical distance between points $v^{13}$ and $v^{24}$ is easily computed to be $\adiff_A (34;12)$.
 Computations are similar  if the type of line $L_{12}$ is $[14,23]$. This proves the second statement.

%
\end{proof}

In the previous theorem,
\emph{the $2\times 2$ tropical minors of $A$ involving three or four different indices} come into play. There are 30  such minors in $A$.
\textbf{Assume that  all the $2\times 2$  minors of $A$ involving three or four different indices are tropically regular} (so that they have non--zero adiffs!). We  call this  \textbf{hypothesis 2}.

\begin{lem}\label{lem:maximal}
Assume hypotheses 1  and 2 for an order 4 matrix $A$. Then $A$ is maximal and, moreover, all $2\times 2$ minors of $A$ are tropically regular.
\end{lem}
\begin{proof}
Suppose $\{i,j,k,l\}=[4]$. Then $L_{ij}\cap L_{ik}=\{\underline{i}\}$ and $L_{ij}\cap L_{kl}=\emptyset$, by hypothesis 1. Then by  theorem \ref{thm:adiffs},  $A$ is maximal.

Now, consider a minor involving only two indices, say $A(12,12)$. It is tropically singular if and only if $a_{12}=a_{21}=0$. If this is the case then, by lemma \ref{lem:dos_puntos},  $L_{12}$ provides less that four extremals to $\spann(A)$, so that $A$ is not maximal.
\end{proof}

\begin{rem}\label{rem:open_q}
Let  four points in $\TP$ be given as the columns of a matrix $C$. The points are in \textbf{tropical general position} if, by definition,  all the $k\times k$ minors of $C$ are tropically regular, for all $2\le k\le 4$.

If $A$  satisfies hypotheses 1 and 2, then $A$ is  tropically regular, by \cite{Richter}, and all the $2\times 2$  minors of $A$ are tropically regular, by lemma \ref{lem:maximal}.  But, are all the $3\times 3$  minors of $A$  tropically regular, i.e., are the points represented by the columns of $A$ in general position? We cannot answer this question yet.
\end{rem}


\begin{ex}
Here is a convex symmetric non--maximal example.
\begin{equation}\label{eqn:tambor}
A=\left[\begin{array}{rrrr}
0&-4&-6&-10\\
-4&0&-10&-6\\
-6&-10&0&-4\\
-10&-6&-4&0
\end{array}\right].
\end{equation}
The f--vector of $\spann(A)$ is $(8,14,8)$, the polygon--vector is $(4,4,0,0)$. Lines $L_{13}$, $L_{14}$, $L_{23}$ and $L_{24}$ are of type $[12,34]$ and  $L_{12}$ and $L_{34}$ are of type $[13,24]$; see figure \ref{figure_11}.
\end{ex}

\begin{figure}[h]
 \centering
  \includegraphics[keepaspectratio,width=14cm]{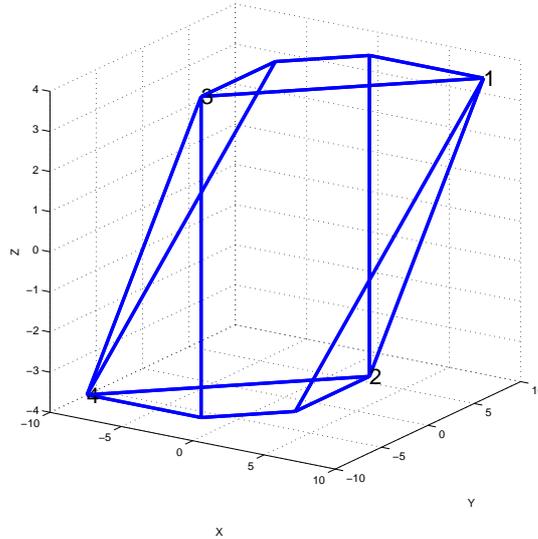}\\
  \caption{The tropical tetrahedron given by matrix (\ref{eqn:tambor}).}
  \label{figure_11}
\end{figure}

\subsection{Towards a classification of the combinatorial types of $\spann(A)$}
From now on, we will assume that $A$ satisfies hypotheses 1 and 2.

\label{dfn:ten_points}
Choose an oddly generated extremal in $\spann(A)$  and look at the configuration at this point: we know that it is either  $(5.5.5)$ right or left or $(p.q.r)$, with $\{p,q,r\}=\{4,5,6\}$. 
The following lemmas tell us that this configuration is encoded in the types of the 3 tropical lines passing through it.

\begin{lem}\label{lem:123right}
Assume hypotheses 1 and 2.
The following are equivalent:
 \begin{enumerate}
 \item the  point \underline{123} is $(5.5.5)$ right,
\item  $L_{12}$ is of type $[14,23]$, $L_{23}$ is of type $[24,13]$ and $L_{31}$ is of type $[34,12]$.
 \end{enumerate}
 In such a case, the table below shows triples of points and the equations of the classical planes they classically span:
 $$\begin{tabular}{c|c}
 points&equation\\
 \hline
 \underline{2} \underline{12} \underline{21}&$X_3-X_2=a_{32}$\\
 \underline{3} \underline{23} \underline{32}&$X_1-X_3=a_{13}$\\
 \underline{1} \underline{31} \underline{13}&$X_2-X_1=a_{21}$.\\
 \end{tabular}$$
\end{lem}
\begin{proof}
 From the right--hand--side of figure \ref{figure_04} we conclude that
  the direction of the bounded edge of line $L_{12}$ is $e_{23}$, of line $L_{23}$ is $e_{31}$ and of line $L_{31}$ is $e_{12}$.
The equivalence follows from here.

Next, recall  from elementary linear  algebra that the equation of the plane generated by three points $p,q,r$ in $\R^3$ is given by setting
a certain $4\times 4$ determinant, denoted $D(p,q,r)$, equal to zero:
the first three rows of $D(p,q,r)$ are filled in with the coordinates of $p,q,r$ and the coordinates of an indeterminate point  $(X,X_2,X_3)^t$ and
the fourth row  is full of ones. The coordinates of \underline{12} and \underline{21} are easily obtained either from expressions
(\ref{eqn:r12}) and (\ref{eqn:r34}) or from figure \ref{figure_04} and theorem \ref{thm:adiffs}. Now, we compute and factor the classical determinant $D(\underline{2}, \underline{12}, \underline{21})$, obtaining
\begin{eqnarray*}
&\left|\begin{array}{rrrr}
a_{12}-a_{42}&-a_{41}&-a_{41}&X_1\\
-a_{42}&a_{31}-a_{41}-a_{32}&-a_{42}&X_2\\
a_{32}-a_{42}&a_{31}-a_{41}&a_{32}-a_{42}&X_3\\
1&1&1&1\\
\end{array}\right|\\
=&\pm \adiff_A (34;12)\adiff_A (14; 24)(X_2-X_3+a_{32}).\\
\end{eqnarray*}
By hypothesis 2, the adiffs in the line right above are non--zero, so that the equation reduces to  $X_3-X_2=a_{32}$. The computations are similar for other classical planes.
\end{proof}

\begin{lem}\label{lem:123left}
Assume hypotheses 1 and 2 for $A$.
The following are equivalent:
 \begin{enumerate}
 \item the  point \underline{123} is $(5.5.5)$ left,
\item  $L_{12}$ is of type $[13,24]$, $L_{23}$ is of type $[12,34]$ and $L_{31}$ is of type $[14,23]$.
\end{enumerate}
In such a case, the table below shows triples of points and the equations of the classical planes they classically span:
 $$\begin{tabular}{c|cc}
 points&equation\\
 \hline
 \underline{1} \underline{12} \underline{21}&$X_3-X_1=a_{31}$\\
 \underline{2} \underline{23} \underline{32}&$X_1-X_2=a_{12}$\\
 \underline{3} \underline{31} \underline{13}&$X_2-X_3=a_{23}$.&\qed\\
 \end{tabular}$$
 \end{lem}

\qed\begin{lem}\label{lem:4right}
Assume hypotheses 1 and 2 for $A$.
The following are equivalent:
 \begin{enumerate}
 \item the  point \underline{4} is $(5.5.5)$ right,
\item  $L_{14}$ is of type $[12,34]$, $L_{24}$ is of type $[14,23]$ and $L_{34}$ is of type $[13,24]$.
\end{enumerate}
In such a case, the table below shows triples of points and the equations of the classical planes they classically span:
 $$\begin{tabular}{c|cc}
 points&equation\\
 \hline
 \underline{1} \underline{14}  \underline{41}&$X_2-X_1=a_{21}$\\
 \underline{2} \underline{24}  \underline{42}&$X_3-X_2=a_{32}$\\
 \underline{3} \underline{34}  \underline{43}&$X_1-X_3=a_{13}$.&\qed\\
 \end{tabular}$$
 \end{lem}

\begin{lem}\label{lem:4left}
Assume hypotheses 1 and 2 for $A$.
The following are equivalent:
 \begin{enumerate}
 \item the  point \underline{4} is $(5.5.5)$ left,
\item  $L_{14}$ is of type $[13,24]$, $L_{24}$ is of type $[12,34]$ and $L_{34}$ is of type $[14,23]$.
 \end{enumerate}
 In such a case, the table below shows triples of points and the equations of the classical planes they classically span:
 $$\begin{tabular}{c|cc}
 points&equation\\
 \hline
 \underline{1} \underline{14}  \underline{41}&$X_3-X_1=a_{31}$\\
 \underline{2} \underline{24}  \underline{42}&$X_1-X_2=a_{12}$\\
 \underline{3} \underline{34}  \underline{43}&$X_2-X_3=a_{23}$.&\qed\\
 \end{tabular}$$
\end{lem}

Similarly, looking at figure \ref{figure_07}, we can prove
\begin{itemize}
\item \underline{123} is $(4.6.5)$ if and only if $L_{12}$ is $[13,24]$, $L_{23}$ is $[13,24]$ and $L_{31}$ is $[12,34]$,
\item \underline{123} is $(6.4.5)$ if and only if $L_{12}$ is $[23,14]$, $L_{23}$ is $[12,34]$ and $L_{31}$ is $[23,14]$.
\end{itemize}

Analogous statements can be proved for \underline{123} and \underline{4} with any configuration $(p.q.r)$, with $\{p,q,r\}=\{4,5,6\}$.

\m
By the lemmas and comment above, the types of $L_{12}$, $L_{23}$, $L_{31}$ (resp. $L_{14}$, $L_{24}$, $L_{34}$) determine the configuration at point $\underline{123}$ (resp. $\underline{4}$) and conversely. This leads us to
define the \label{dfn:type--vector} \textbf{type--vector}  $t=(t_1,t_2,t_3)$,  where $t_j$ is the number of tropical lines   of type $[4j,kl]$, with $\{j,k,l\}=\{1,2,3\}$.  By  theorem \ref{thm:types}, both for \underline{123} and \underline{4}, either the three types are  all  different or just two of them are equal.\label{par:lines} Thus, up to a permutation, $t$ equals one of the following
\begin{equation}
(2,2,2),\quad(3,2,1),\quad(3,3,0),\quad(4,1,1),\quad(4,2,0).
\end{equation}

\m
\textbf{Hexagons in $\spann(A)$.}
It is easy to realize that two lines of the same type  (say $[ik,jl]$) having concatenated indices (say $L_{ij}$, $L_{jk}$) yield one hexagon in $\spann(A)$ at extremal $\underline{ijk}$. However, two lines of the same type  having disjoint indices (say $L_{ij}$, $L_{kl}$, with $\{i,j,k,l\}=[4]$) yield no hexagon at all.
And, what happens if two or more hexagons are adjacent facets of $\spann(A)$? Fix a type, e.g. $[ik,jl]$.
\begin{itemize}
\item  Three lines of the same type  having concatenated indices (say $L_{ij}$, $L_{jk}$, $L_{kl}$) yield two adjacent hexagons. The converse is true.
\item  Four lines of the same type  necessarily have concatenated indices (say $L_{ij}$, $L_{jk}$, $L_{kl}$, $L_{li}$) and yield four adjacent hexagons closing up into a cycle. The converse is true.
\end{itemize}

\subsection{Searching for $\spann(A)$ with the combinatorial type of a regular dodecahedron}\label{subsec:dodecahedra}

We seek an order 4 Kleene star  matrix $A$ having f--vector $(20,30,12)$ and polygon--vector $(0,0,12,0)$. We will not find any!

By a translation and a change of coordinates, we can assume that the coordinates of \underline{123} and \underline{4}
are $[0,0,0,0]^t$ and $[-a,-b,-c,0]^t$, respectively, with $0<a\le b\le c$. If $\spann(A)$ must have 12 pentagonal facets, then we can assume that points \underline{123} and \underline{4} are
both $(5.5.5)$; this way  $\spann(A)$ has, at least, 6 pentagonal facets.
Four cases arise, depending on whether the points  \underline{123} and \underline{4} are left or right. These are dealt with in theorems \ref{thm:left_left}, \ref{thm:right_right} and \ref{thm:left_right}.

\begin{thm}\label{thm:left_left}
If  \underline{123} and \underline{4} are both $(5.5.5)$ left,  then polygon--vector of $\spann(A)$ is $(0,3,6,3)$.
\end{thm}

\begin{proof}
By lemmas  \ref{lem:123left} and \ref{lem:4left}, the type--vector is $(2,2,2)$ and
\begin{itemize}
\item  points \underline{1}, \underline{14}, \underline{41}, \underline{$12$}, \underline{$21$} lie on the classical plane $X_3-X_1=a_{31}$,
\item  points \underline{2}, \underline{24}, \underline{42}, \underline{$23$}, \underline{$32$} lie on the classical plane $X_1-X_2=a_{12}$,
\item  points \underline{3}, \underline{34}, \underline{43}, \underline{$31$}, \underline{$13$} lie on the classical plane $X_2-X_3=a_{23}$.
\end{itemize}
Moreover, the coordinates of $\underline{134}=\col((-A^t)_0,2)$ also satisfy  the equation $X_2-X_3=a_{23}$, so that \underline{3}, \underline{34}, \underline{43}, \underline{$31$}, \underline{$13$} and \underline{$134$} make up a hexagon.  Similar for  the points \underline{$124$}, \underline{$234$},  and so $\spann(A)$ has three hexagons, three pentagons and three quadrangles.
\end{proof}

\begin{thm}\label{thm:right_right}
If  \underline{123} and \underline{4} are both $(5.5.5)$ right,  then polygon--vector of $\spann(A)$ is $(0,3,6,3)$.\qed
\end{thm}

\begin{thm}\label{thm:left_right}
It is not possible to have $A$ satisfying hypotheses 1 and 2 such that \underline{123} is $(5.5.5)$ left and \underline{4} is $(5.5.5)$ right or \underline{123} is $(5.5.5)$ right and \underline{4} is $(5.5.5)$ left.
\end{thm}

\begin{proof} 
By symmetry, it is enough to address the  case where \underline{123} is $(5.5.5)$ left and \underline{4} is $(5.5.5)$ right.
If \underline{123} is $(5.5.5)$ left then, using lemma \ref{lem:123left}, we have
\begin{equation*}
\underline{12}=\colur{-a_{41}}{-a_{42}}{a_{31}-a_{41}}{0}, \underline{23}=\colur{a_{12}-a_{42}}{-a_{42}}{-a_{43}}{0},
\underline{31}=\colur{-a_{41}}{a_{23}-a_{43}}{-a_{43}}{0},
\end{equation*}
\begin{equation*}
\underline{21}=\colur{a_{32}-a_{42}-a_{31}}{-a_{42}}{a_{32}-a_{42}}{0}, \underline{32}=\colur{a_{13}-a_{43}}{a_{13}-a_{43}-a_{12}}{-a_{43}}{0},
\underline{13}=\colur{-a_{41}}{a_{21}-a_{41}}{a_{21}-a_{41}-a_{23}}{0}.
\end{equation*}

Since \underline{4} is $(5.5.5)$ right then, using lemma \ref{lem:4right}, we have
\begin{equation*}
\underline{14}=\colur{a_{34}-a_{31}}{a_{34}-a_{31}+a_{21}}{a_{34}}{0}, \underline{24}=\colur{a_{14}}{a_{14}-a_{12}}{a_{14}-a_{12}+a_{32}}{0},
\underline{34}=\colur{a_{24}-a_{23}+a_{13}}{a_{24}}{a_{24}-a_{23}}{0},
\end{equation*}
\begin{equation*}
\underline{41}=\colur{a_{24}-a_{21}}{a_{24}}{a_{34}}{0}, \underline{42}=\colur{a_{14}}{a_{34}-a_{32}}{a_{34}}{0},
\underline{43}=\colur{a_{14}}{a_{24}}{a_{14}-a_{13}}{0}.
\end{equation*}

We can assume that the coordinates of \underline{123} are $[0,0,0,0]^t$ and the coordinates of \underline{4} are $[-a,-b,-c,0]^t$, for
some positive $a,b,c$.
Then, by expression (\ref{eqn:3penta_left}),
\begin{equation}\label{eqn:Aleft}
 A=A_0=\left[
 \begin{array}{rrrr}
 0&-\gamma_2-\delta_2&-\gamma_3-\gamma_2-\delta_2&-a\\
 -\gamma_1-\gamma_3-\delta_3&0&-\gamma_3-\delta_3&-b\\
 -\gamma_1-\delta_1&-\gamma_2-\gamma_1-\delta_1&0&-c\\
 0&0&0&0
 \end{array}
 \right].
 \end{equation}

 Substituting $a_{ij}$ by its value in the coordinates of \underline{14}, we get
\begin{equation*}
\underline{14}=\colur{-c+\gamma_1+\delta_1}{-c+\delta_1-\gamma_3-\delta_3}{-c}{0},
\end{equation*}
 so that the vector $\overrightarrow{\underline{14}\ \underline{1}}$ equals $(c-\gamma_1-\delta_1)e_{123}$. Similarly, we see that the vectors $\overrightarrow{\underline{24}\ \underline{2}}$ and $\overrightarrow{\underline{34}\ \underline{3}}$ have direction $e_{123}$.

 We use lemma \ref{lem:4right} and theorem \ref{thm:adiffs} to obtain
\begin{itemize}\label{eqn:distances}
\item $\tdist (\underline{14}, \underline{41})=\adiff_A  (23,14)=|-\gamma_3-\delta_3+\delta_1-c+b|>0$,
\item $\dist (\underline{1}, \underline{14})=\adiff_A  (34,14)=|\gamma_1+\delta_3-c|>0$,
\item $\dist (\underline{4}, \underline{41})=\adiff_A  (12,14)=|\gamma_1+\gamma_3+\delta_3+a-b|>0$,

\item $\tdist (\underline{24}, \underline{42})=\adiff_A (13,24)=|-\gamma_1-\delta_1+\delta_2-a+c|>0$,
\item $\dist (\underline{2}, \underline{24})=\adiff_A (14,24)=|\gamma_2+\delta_2-a|>0$,
\item $\dist (\underline{4}, \underline{42})=\adiff_A (23,24)=|\gamma_1+\delta_1+\gamma_2+b-c|>0$,

\item $\tdist (\underline{34}, \underline{43})=\adiff_A (12,34)=|-\gamma_2-\delta_2+\delta_3-b+a|>0$,
\item $\dist (\underline{3}, \underline{34})=\adiff_A (24,34)=|\gamma_3+\delta_3-b|>0$,
\item $\dist (\underline{4}, \underline{43})=\adiff_A (13,34)=|\gamma_2+\delta_2+\gamma_3+c-a|>0$.
\end{itemize}

For each $j=1,2,3$, it is obvious that
\begin{equation}\label{eqn:j}
\underline{4}+\overrightarrow{\underline{4}\ \underline{4j}}+\overrightarrow{\underline{4j}\ \underline{j4}}+\overrightarrow{\underline{j4}\ \underline{j}}=\underline{j};
\end{equation}
see figure \ref{figure_12}, for $j=1$.
\begin{figure}[h]
 \centering
  \includegraphics[keepaspectratio,width=14cm]{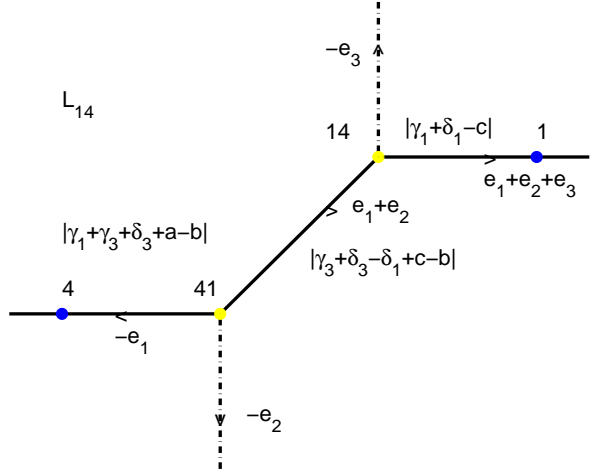}\\
  \caption{Going from point \underline{4} to point \underline{1} along the tropical line $L_{14}$.}
  \label{figure_12}
\end{figure}

Having in mind  tropical distances and the directions of the vectors $e_i$, $e_{ij}$, $e_{123}$, equalities (\ref{eqn:j})
convert into the following ones
\begin{equation}\label{eqn:j1}
\colur
{-a+|\gamma_1+\gamma_3+\delta_3+a-b|+|-\gamma_3-\delta_3+\delta_1-c+b|+|\gamma_1+\delta_1-c|}
{-b+|-\gamma_3-\delta_3+\delta_1-c+b|+|\gamma_1+\delta_1-c|}
{-c+|\gamma_1+\delta_1-c|}
{0}=
\colur
{0}
{-\gamma_1-\gamma_3-\delta_3}
{-\gamma_1-\delta_1}
{0}
\end{equation}
\begin{equation}\label{eqn:j2}
\colur
{-a+|\gamma_2+\delta_1-a|}
{-b+|\gamma_1+\delta_1+\gamma_2+b-c|+|-\gamma_1-\delta_1+\delta_2-a+c|+|\gamma_2+\delta_1-a|}
{-c+|-\gamma_1-\delta_1+\delta_2-a+c|+|\gamma_2+\delta_1-a|}
{0}=
\colur
{-\gamma_2-\delta_2}
{0}
{-\gamma_2-\gamma_1-\delta_1}
{0}
\end{equation}
\begin{equation}\label{eqn:j3}
\colur
{-a+|\gamma_2+\delta_2-\delta_3+b-a|+|\gamma_3+\delta_3-b|}
{-b+|\gamma_3+\delta_3-b|}
{-c+|\gamma_2+\delta_2+\gamma_3+c-a|+|\gamma_2+\delta_2-\delta_3+b-a|+|\gamma_3+\delta_3-b|}
{0}=
\colur
{-\gamma_3-\gamma_2-\delta_2}
{-\gamma_3-\delta_3}
{0}
{0}
\end{equation}
Working through equalities (\ref{eqn:j1}),(\ref{eqn:j2}) and (\ref{eqn:j3}), and using
 that no absolute value vanishes  (due to maximality of $A$), it follows that
\begin{align*}
\gamma_1+\delta_1&<c,\\
\gamma_2+\delta_2&<a,\\
\gamma_3+\delta_3&<b,\\
\gamma_1+\delta_1-\delta_2&<c-a,\\
\gamma_2+\delta_2-\delta_3&<a-b,\\
\gamma_3+\delta_3-\delta_1&<b-c,
\end{align*}
but then $\gamma_1+\gamma_2+\gamma_3<0$, contradicting $\gamma_j>0$, for all $j=1,2,3$.

 \end{proof}

 \subsection{Classification}

For $\spann(A)$ convex and maximal, here is a list of \textbf{additional properties} it enjoys (with a short explanation  provided):
\begin{enumerate}
\item  each facet of $\spann(A)$ contains exactly one generator, exactly one 3--generated  extremal and 1, 2 or 3 2--generated extremals, (this is true for facets meeting  $\underline{123}$ in $\spann(A)$, thus it is true for facets meeting  $\underline{4}$ in $\spann(A^t)$, 
    thus it is true for facets  meeting any oddly generated extremal, which are all of the facets),\label{caso:a}
\item at an oddly generated extremal no two  hexagons meet, (same reasons as in item above),\label{caso:c}

\item  $f_4=f_6$, i.e., the number of quadrangles in $\spann(A)$ equals the number of hexagons, (this is true because the f--vector of $\spann(A)$ is $(20,30,12)$, the same f--vector  the regular dodecahedron $\cal D$ has. Since $\spann(A)$ is trivalent (by  \cite{Joswig_K}), then the combinatorial type of $\spann(A)$ can be obtained from $\cal D$ by a finite sequence of  combinatorial polyhedral transformations; see \cite{Diaz})\label{caso:g}.
\end{enumerate}

\begin{cor}\label{cor:no(0,0,12,0)ni(0,1,10,1)}
It is not possible to have $A$ 
with f--vector  $(20,30,12)$ and  polygon--vector  either $(0,0,12,0)$ or $(0,1,10,1)$. In particular, $\spann(A)$ does not have the combinatorial type of the regular dodecahedron $\cal D$, for any $A$.
\end{cor}

\begin{proof}
Assume that $A$ satisfies hypotheses 1 and 2, has f--vector $(20,30,12)$, $f_4=f_6$ and  $f_5\ge10$.
There are three cases:
\begin{itemize}
\item all the facets of $\spann(A)$ are pentagons,i.e, $f_5=12$,
\item a quadrangle and a hexagon are adjacent facets in $\spann(A)$,
\item a quadrangle and a hexagon are non--adjacent facets in $\spann(A)$.
\end{itemize}
In the third case, $6+4=10$ extremals  of $\spann(A)$ meet either a quadrangle or a hexagon  and $f_5\ge10$, so that the remaining 10 extremals  in $\spann(A)$ must have configuration $(5.5.5)$. Then,  there must exist two opposite extremals (recall definition in p. \pageref{dfn:opposite})   of $\spann(A)$ both of which are $(5.5.5)$. But then, theorems \ref{thm:left_left}, \ref{thm:right_right} and \ref{thm:left_right} tell us that this cannot happen.
 In the first two cases, the existence of two opposite extremals having configuration $(5.5.5)$ is even more obvious.
\end{proof}

By corollary \ref{cor:no(0,0,12,0)ni(0,1,10,1)} and item \ref{caso:a}, we have $2\le f_6\le 4$. \label{caso:f}
Thus, the polygon--vector $(0,f_4,f_5,f_6)$ of $\spann(A)$ is
\begin{equation}
(0,2,8,2),\qquad (0,3,6,3),\qquad (0,4,4,4).
\end{equation}

Our classification goes according to the type--vector (see p. \pageref{dfn:type--vector}) and number and adjacency of hexagons (the polygon--vector is not enough to classify!).
Up to symmetry and changes of coordinates, the combinatorial type of $\spann(A)$ is classified as follows:

\setcounter{enumi}{0}

\begin{list}{Class \theenumi.}{\usecounter{enumi}}
\item \label{class:first} If $t=(2,2,2)$, then by  theorems \ref{thm:adiffs} and \ref{thm:left_right}, the points $\underline{123}$ and $\underline{4}$ are both $(5.5.5)$ left or both $(5.5.5)$ right. 
    In any case, the polygon--vector is $(0,3,6,3)$, by theorems \ref{thm:left_left}   and \ref{thm:right_right}. No pairs of adjacent hexagons exist in $\spann(A)$, since the indices of lines corresponding to any given type are disjoint. For examples, see lemma \ref{lem:gamma_delta_c}.
\item  \label{class:second} If $t=(3,2,1)$ and the indices  of the lines of type $[13,24]$ are concatenated, then $\spann(A)$ has $2+1=3$ hexagons, so that the polygon--vector is $(0,3,6,3)$. Two hexagons are adjacent. See example \ref{ex:AA191}.
\item  \label{class:third} If $t=(3,2,1)$ and the indices  of the lines of type $[13,24]$ are disjoint, then $\spann(A)$ has 2 hexagons, which  are adjacent. The polygon--vector is $(0,2,8,2)$. See example \ref{ex:B15_y_B12}.
\item  \label{class:fourth} If $t=(3,3,0)$, then   $\spann(A)$ has 2 pairs of adjacent hexagons. The polygon--vector is $(0,4,4,4)$. For examples, see lemma \ref{lem:circulant}.
\item  \label{class:fifth} If $t=(4,1,1)$, then   $\spann(A)$ has a cycle of 4 adjacent hexagons. The polygon--vector is $(0,4,4,4)$.  For examples, see lemma \ref{lem:circulant}. In this case, the configurations at $\underline{123}$ and $\underline{4}$ are not equal: one is $(p.q.r)$ and the other is $(r.q.p)$, for some $\{p,q,r\}=\{4,5,6\}$.
\item \label{class:last}  If $t=(4,2,0)$, then   $\spann(A)$ has a cycle of 4 adjacent hexagons. The indices  of the two lines of type $[13,24]$ are disjoint. The polygon--vector is $(0,4,4,4)$. In this case, the configurations at $\underline{123}$ and $\underline{4}$ are equal to $(p.q.r)$, for some $\{p,q,r\}=\{4,5,6\}$. For examples, see lemma \ref{lem:circulant}.
\end{list}
\textbf{Symmetry and chirality.}
The symmetric image of configuration $(5.5.5)$ left is $(5.5.5)$ right. The symmetric image of configuration $(p.q.r)$  is $(p.q.r)$.
Thus, $\spann(A)$ admits a central  symmetry  only  for class \ref{class:last}.

If in  $\spann(A)$, the points \underline{123} and \underline{4} are both $(5.5.5)$ right, then in $\spann(A^t)$, the points \underline{123} and \underline{4} are both $(5.5.5)$ left and, thus, the solids $\spann(A), \spann(A^t)$ are chiral to each other. This happens in class \ref{class:first}.

\subsection{Compatible configurations at \underline{123} and at \underline{4} and examples}\label{subsec:examples}

Fix an order 4 Kleene star  $A$.
As we saw in the proof of theorem \ref{thm:types}, the type of the line $L_{ij}$ depends on the value attained by the tropical minor $M(ij):=|A(kl,ij)|_{\trop}$, where $\{i,j,k,l\}=[4]$. More explicitly,
\begin{equation}\label{eqn:12}
M(12)=\max\{a_{31}+a_{42},a_{41}+a_{32}\}=|A(34,12)|_{\trop},
\end{equation}
\begin{equation}\label{eqn:13}
M(13)=\max\{a_{21}+a_{43},a_{41}+a_{23}\}=|A(24,13)|_{\trop},
\end{equation}
\begin{equation}\label{eqn:14}
M(14)=\max\{a_{21}+a_{34},a_{31}+a_{24}\}=|A(23,14)|_{\trop},
\end{equation}
\begin{equation}\label{eqn:23}
M(23)=\max\{a_{12}+a_{43},a_{42}+a_{13}\}=|A(14,23)|_{\trop},
\end{equation}
\begin{equation}\label{eqn:24}
M(24)=\max\{a_{12}+a_{34},a_{32}+a_{14}\}=|A(13,24)|_{\trop},
\end{equation}
\begin{equation}\label{eqn:34}
M(34)=\max\{a_{13}+a_{24},a_{23}+a_{14}\}=|A(12,34)|_{\trop}.
\end{equation}

\begin{figure}[h]
 \centering
  \includegraphics[keepaspectratio,width=14cm]{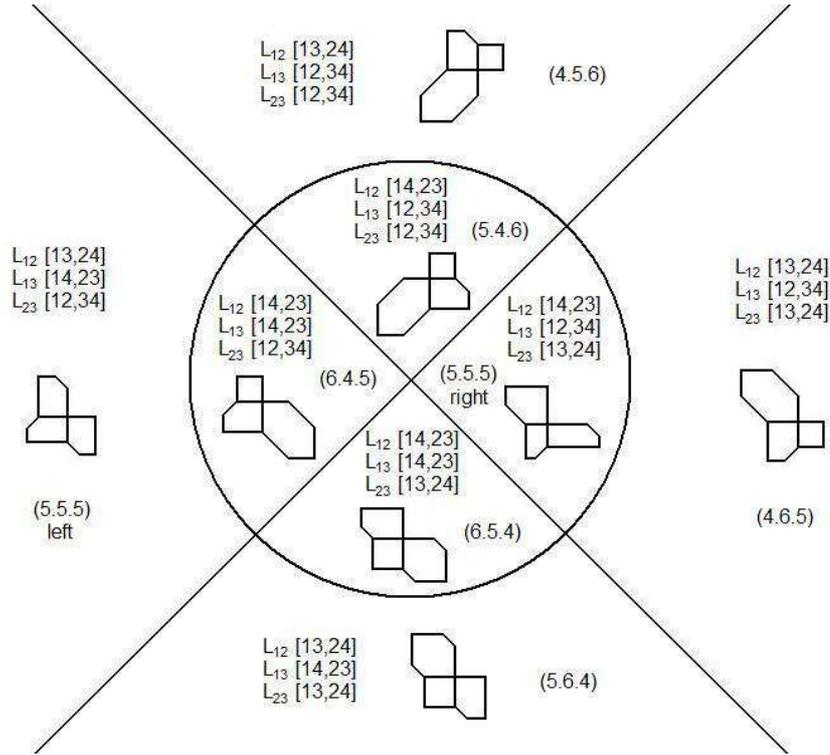}\\
  \caption{Possible configurations at point \underline{123}. The circle and the two straight lines  represent the classical hyperplanes $a_{31}+a_{42}=a_{41}+a_{32}$, $a_{21}+a_{43}=a_{41}+a_{23}$ and $a_{12}+a_{43}=a_{42}+a_{13}$.}
  \label{figure_13}
\end{figure}
\begin{figure}[h]
 \centering
  \includegraphics[keepaspectratio,width=14cm]{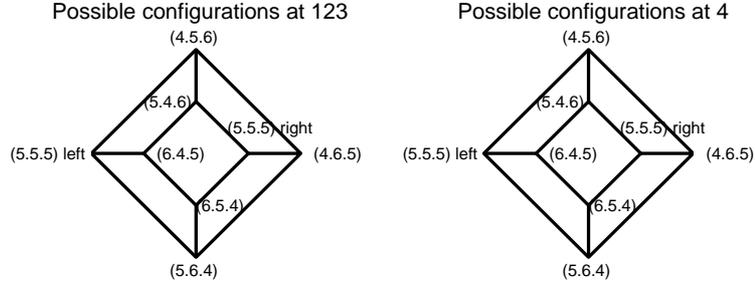}\\
  \caption{Dual graphs of cell decompositions at $\underline{123}$ and at $\underline{4}$.}
  \label{figure_14}
\end{figure}

By inequalities (\ref{eqn:Kleene}), the order 4 Kleene stars  form  a closed convex subset  $\overline{\SSS}$ of $\R_{\le0}^{12}$. Those matrices satisfying hypotheses 1 and 2 form an open dense subset $\SSS$ in $\overline{\SSS}$. For any $A\in\SSS$, the combinatorial type of $\spann(A)$ changes whenever the value of a maximum among (\ref{eqn:12})--(\ref{eqn:34}) changes. Therefore, the following family of hyperplanes  splits $\SSS$ into open cells:
\begin{eqnarray*}
a_{31}+a_{42}&=&a_{41}+a_{32}\\
a_{21}+a_{43}&=&a_{41}+a_{23}\\
a_{21}+a_{34}&=&a_{31}+a_{24}\\
a_{12}+a_{43}&=&a_{42}+a_{13}\\
a_{12}+a_{34}&=&a_{32}+a_{14}\\
a_{13}+a_{24}&=&a_{23}+a_{14}.
\end{eqnarray*}

The situation around point \underline{123} is depicted in figure \ref{figure_13}.  The situation around point \underline{4} is similar. The dual graphs of both cell decompositions appear in figure \ref{figure_14}; each node of the graph on the left (resp. right) corresponds to a possible configuration at $\underline{123}$ (resp. at $\underline{4}$).

We say that  given configurations at $\underline{123}$ and  at $\underline{4}$ are \textbf{compatible}  if there exists  $A\in\SSS$ such that $\spann(A)$ realizes both. To obtain examples of some  combinatorial types of $\spann(A)$, it is enough to find certain compatible pairs of configurations at $\underline{123}$ and  at $\underline{4}$. Theorems \ref{thm:left_left} and \ref{thm:right_right} show two compatible configurations (symmetric one to another), while theorem \ref{thm:left_right} shows some incompatible configurations. 

For each class we can always find examples with integer matrices (due to translations and scaling). Then, the tropical distance between neighboring extremals is the integer length  of the edge they span.

\m
Revisiting  theorems \ref{thm:left_left} and \ref{thm:right_right}, we find matrices $A\in\SSS$ with type--vector $(2,2,2)$, thus in class \ref{class:first}. The following provides simple examples: e.g., take $\gamma=\delta=1$, $c=2$.

\begin{lem}\label{lem:gamma_delta_c}
Suppose that $\gamma,\delta,c$ are positive reals  such that $2\gamma+\delta<2c$. Then
\begin{equation}\label{eqn:gd}
A=\left[\begin{array}{rrrr}
0&-2\gamma-\delta&-\gamma-\delta&-c\\
 -\gamma-\delta&0&-2\gamma-\delta&-c\\
 -2\gamma-\delta&-\gamma-\delta&0&-c\\
 -c&-c&-c&0\\
 \end{array}\right]
\end{equation}
 $\spann(A)$ belongs to class \ref{class:first}.
\end{lem}
\begin{proof}
The matrix   $A$ is normal and satisfies inequalities (\ref{eqn:Kleene}), so that it is tropically idempotent. By expression (\ref{eqn:3penta_right}), the point \underline{123} is $(5.5.5)$ right, with
$\gamma_j=\gamma$, $\delta_j=\delta$, for $j=1,2,3$.

Direct computations yield that lines $L_{13}$ and $L_{14}$  are of type $[12,34]$, lines $L_{23}$ and $L_{34}$  are of type $[13,24]$ and lines $L_{12}$ and $L_{24}$  are of type $[14,23]$. This implies that the point \underline{4} is $(5.5.5)$ right, and then the result  follows from theorem \ref{thm:right_right}.
\end{proof}

We can make additional computations for matrix (\ref{eqn:gd}). Lemma \ref{lem:transpose} provides the coordinates of the points \underline{$ij4$}. We compute  the coordinates of the  points below, obtaining
\begin{equation}
\underline{14}=\colur{2\gamma+\delta-c}{\gamma-c}{-c}{0}, \quad \underline{41}=\colur{\gamma+\delta-c}{-c}{-c}{0},
\end{equation}
\begin{equation}
\underline{24}=\colur{-c}{2\gamma+\delta-c}{\gamma-c}{0}, \quad \underline{42}=\colur{-c}{\gamma+\delta-c}{-c}{0},
\end{equation}
\begin{equation}
\underline{34}=\colur{\gamma-c}{-c}{2\gamma+\delta-c}{0}, \quad \underline{43}=\colur{-c}{-c}{\gamma+\delta-c}{0},
\end{equation}
showing that \underline{4}, \underline{234}, \underline{24}, \underline{42}, \underline{43} lie on $X_1=-c$, \underline{4}, \underline{134}, \underline{34}, \underline{43}, \underline{41} lie on $X_2=-c$ and \underline{4}, \underline{124}, \underline{14}, \underline{41}, \underline{42} lie on $X_3=-c$.

\m

Given $p=[p_1,p_2,p_3,p_4]^t$, we consider the corresponding \textbf{circulant} and \textbf{anticirculant symmetric} matrices
\begin{equation}
C(p)=\left[\begin{array}{rrrr}
p_1&p_4&p_3&p_2\\
p_2&p_1&p_4&p_3\\
p_3&p_2&p_1&p_4\\
p_4&p_3&p_2&p_1
\end{array}\right],\qquad A(p)=\left[\begin{array}{rrrr}
p_1&p_2&p_3&p_4\\
p_2&p_1&p_4&p_3\\
p_3&p_4&p_1&p_2\\
p_4&p_3&p_2&p_1
\end{array}\right].
\end{equation}


In the lemma below, take for instance, $a=3, c=6$ and  $b=4$ or $b=5$.

\begin{lem}\label{lem:circulant}
Consider  numbers  $a,b,c$ such that $0<a<b<c\le 2a$ and $2b\neq a+c$ and set $p=[0,-a,-b,-c]^t$. Then
\begin{enumerate}
\item $\spann C(p)$   belongs to classes \ref{class:fourth} or \ref{class:fifth},\label{ite:uno}
\item $\spann A(p)$ belongs to the class \ref{class:last}.\label{ite:dos}
\end{enumerate}
\end{lem}

\begin{proof} The proofs of both items are similar. We will only prove item \ref{ite:uno}.
The hypothesis on $a,b,c$ guarantee that $C(p)$ is a Kleene star and all the
 $2\times 2$ minors of $C(p)$ are tropically regular. Thus, $C(p)$ satisfies hypothesis 2.
 Concerning hypothesis 1, notice that the vertex of $\Pi:=L(\underline{1},\underline{2},\underline{3})$ is $\col(-C(p)^t,4)=[c,b,a,0]^t$, so that the tropical linear form corresponding to $\Pi$ is
\begin{equation*}
-c\odot X_1\oplus -b\odot X_2\oplus -a\odot X_3\oplus 0\odot X_4=\max\{X_1-c, X_2-b,  X_3-a, X_4\}
\end{equation*}
and, plugging in the coordinates of \underline{4}, we get that the maximum
\begin{equation*}
\max\{-a-c, -2b,  -a-c, 0\}=0
\end{equation*}
is attained only once, showing that the columns of $C(p)$ are not coplanar, tropically.

The type of  $L_{13}$ is $[12,34]$ and  the type of  $L_{24}$ is $[14,23]$.
The remaining computations depend on
$$M=\max\{-2b,-a-c\},$$ the value of  any of the $2\times 2$ tropical minors $A(ij,kl)$ with four different indices.

If $M=-2b$, then the type of lines $L_{12}$, $L_{14}$  $L_{23}$ and $L_{34}$ is $[13,24]$. Thus, the type--vector is $(1,4,1)$.
Otherwise, $M=-a-c$ and the type of lines $L_{12}$ and  $L_{34}$ is $[14,23]$, while the type of lines $L_{14}$ and  $L_{23}$ is $[12,34]$. The type--vector is $(3,0,3)$, now. In any case, the polygon--vector is $(0,4,4,4)$.
\end{proof}

Here are more computations for $\spann C(p)$: the coordinates of the points $\underline{ij}$ and the tropical distances between points, obtaining:

\begin{itemize}
\item $\tdist( \underline{13}, \underline{31})=\tdist( \underline{24}, \underline{42})=2(c-a)$,

\item $\dist( \underline{1}, \underline{13})=\dist( \underline{31}, \underline{3})=\dist( \underline{2}, \underline{24})=\dist( \underline{42}, \underline{4})=a+b-c$,

\item $\tdist( \underline{ij}, \underline{ji})=|a-2b+c|$, for other choices of $i,j$,

\item   $\tdist( \underline{i}, \underline{ij})$ are: $b+c-a$, $c+a-b$, $2c-b$, $2c-a$,  $2b-a$, $2b-c$, $2a-b$ and $2a-c$, for other choices of $i,j$.
\end{itemize}
These distances are all strictly positive, by maximality. Moreover,
\begin{equation}
\underline{13}=\colur{c}{c-a}{a}{0}, \underline{31}=\colur{2a-c}{a-c}{a}{0},
\end{equation}
\begin{equation}
\underline{24}=\colur{-a}{c-a}{c-2a}{0}, \underline{42}=\colur{-a}{a-c}{-c}{0},
\end{equation}

Consider  the point \underline{123}, of coordinates $[c,b,a,0]^t$,  and all the  nine  points around it:  \underline{1}, \underline{2}, \underline{3}, \underline{12}, \underline{21}, \underline{13}, \underline{31}, \underline{23} and \underline{32}.

Assume that $M=-2b$.  Computing coordinates, we get
\begin{equation*}
\underline{12}=\colur{c}{b}{c-b}{0}, \underline{21}=\colur{2b-a}{b}{b-a}{0},
\end{equation*}
 \begin{equation*}
\underline{23}=\colur{b-c}{b}{2b-c}{0}, \underline{32}=\colur{a-b}{b}{a}{0}.
\end{equation*}
Thus
\begin{itemize}
\item  points \underline{123}, \underline{1}, \underline{12} and \underline{13} lie in the classical plane of equation $X_1=c$, making a quadrangle,
\item  points \underline{123}, \underline{31}, \underline{13},  \underline{3} and \underline{32} lie in the classical plane of equation $X_3=a$, making a pentagon,
\item  points \underline{123}, \underline{12}, \underline{21},  \underline{2}, \underline{23}, and \underline{32} lie in the classical plane of equation $X_2=b$, making a hexagon.
\end{itemize}

If $M=-a-c$, then
\begin{equation*}
\underline{12}=\colur{c}{a-b+c}{c-b}{0}, \underline{21}=\colur{c}{b}{b-a}{0},
\end{equation*}
 \begin{equation*}
\underline{23}=\colur{b-c}{b}{a}{0}, \underline{32}=\colur{a-b}{a-b+c}{a}{0}.
\end{equation*}
Thus
\begin{itemize}
\item  points \underline{123}, \underline{1}, \underline{12}, \underline{21} and \underline{13}  make a pentagon in $X_1=c$,
\item  points \underline{123}, \underline{23},  \underline{2} and \underline{21}  make a quadrangle in  $X_2=b$,
\item  points \underline{123}, \underline{23}, \underline{32}, \underline{3}, \underline{31}, and \underline{13} make a hexagon in  $X_3=a$.
\end{itemize}

Around the points \underline{124}, \underline{134}, \underline{234} the situation is similar to \underline{123}, by a change of coordinates.

%
\m
Notice that $\spann C(p)$ is not maximal, if $2b=a+c$.

\begin{rem}
If $M=-2b$, then the vertices of the  lines $L_{ij}$ can be arranged  into  two more circulant matrices of vectors $[0,c,b,c-b]^t$ and $[0,a-b,b,a]^t$, with corresponding columns \underline{23}, \underline{34}, \underline{41}, \underline{12}, and  \underline{43}, \underline{14}, \underline{21}, \underline{32}. 

Similarly, for $M=-a-c$, $[0,b-c,b,a]^t$ and $[0,c,b,b-a]^t$,  \underline{34}, \underline{41}, \underline{12}, \underline{23}, and \underline{32}, \underline{43}, \underline{14}, \underline{21}.

\end{rem}


%
%

\begin{figure}[h]
 \centering
  \includegraphics[keepaspectratio,width=14cm]{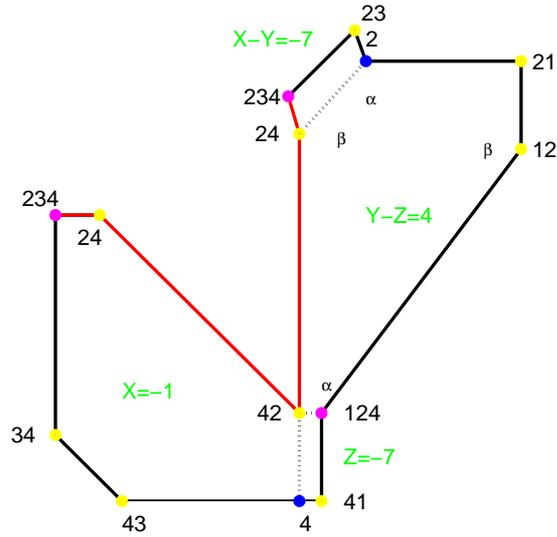}\\
  \caption{Hexagonal adjacent facets of example \ref{ex:B15_y_B12}. Dotted lines must be valley--folded. Red segments must be glued together. Angles $\alpha\simeq125^\circ  16'$ and $\beta\simeq 144^\circ  44'$ are defined in p.\pageref{dfn:alpha}. Other angles in this picture are $45^\circ, 90^\circ$ or $135^\circ$ and $\overline{\alpha}$.}
  \label{figure_15}
\end{figure}



%

\begin{ex} \label{ex:B15_y_B12} For
\begin{equation}
A=\left[\begin{array}{rrrr}
0&-7&-5&-1\\
-8&0&-8&-7\\
-7&-4&0&-7\\
-9&-9&-8&0
\end{array}\right], \qquad A'=\left[\begin{array}{rrrr}
0&-6&-6&-4\\
-4&0&-9&-6\\
-6&-8&0&-7\\
-10&-7&-6&0\\
\end{array}\right]
\end{equation}
 $\spann(A)$, $\spann(A')$ belong both to class \ref{class:third}.
This class is not found in \cite{Joswig_K}.

\m
Let us check the details  for $A$. 
The tropical lines $L_{12}, L_{24}, L_{34}$ are of type $[14,23]$,   $L_{23}, L_{14}$ are of type $[13,24]$ and $L_{13}$ is of type $[12,34]$. Thus, the type--vector is $t=(1,2,3)$. The indices of the lines of type $[13,24]$ are disjoint ($23$ and $14$). Therefore, $\spann(A)$ contains only two hexagonal facets;  see  figure \ref{figure_15}. These are
\underline{2},\underline{12}, \underline{21}, \underline{24}, \underline{42}, \underline{124}
on $X_2-X_3=4$ and
\underline{4}, \underline{24}, \underline{34},\underline{42}, \underline{43}, \underline{234} on $X_1=-1$. The common segment joins \underline{24} to \underline{42}.

The two quadrangular facets are
\underline{2}, \underline{23}, \underline{24}, \underline{234}
on the plane on $X_1-X_2=-7$, and
\underline{4}, \underline{41}, \underline{42}, \underline{124}
on $X_3=-7$.

Notice that at \underline{24},  two hexagons and a quadrilateral meet, and so a total of 12 extremals appear in the configuration of \underline{24}. \label{hex} However, we know that, whatever the configuration
at any oddly generated extremal is,  exactly ten extremals appear in it; see p. \pageref{dfn:ten_points}. 
Therefore, this solid is \textbf{not vertex--transitive}. \label{dfn:no_v_t}

The  angle--vector (see p.\pageref{dfn:angle--vector}) of the hexagon contained in $X_2-X_3=4$ is $\langle 90,\alpha, \beta,90,\alpha,\beta\rangle$, with $\alpha\simeq125^\circ  16'$  and $\beta\simeq 35^\circ 16'$, and $2( 90+\alpha+\beta)=720=180\times(6-2)$ is the sum of the interior angles of a hexagon.
\end{ex}

\begin{ex}\label{ex:AA191}
\begin{equation}
A=\left[\begin{array}{rrrr}
0&-6&-10&-5\\
-6&0&-5&-3\\
-3&-5&0&-6\\
-5&-3&-6&0
\end{array}\right].
\end{equation}
In this example, the type--vector is $(1,2,3)$ and the  polygon--vector is $(0,3,6,3)$, belonging to class 2.
The configuration of the point $\underline{123}$ is $(5.5.5)$ left and the configuration of  $\underline{4}$ is $(6.5.4)$.

We have obtained $A$ as a  perturbation of  the anticirculant symmetric matrix $A(p)$, with $p=[0,-6,-3,-5]^t$. Indeed, in $\spann A(p)$, the configuration is $(5.6.4)$ for points \underline{123} and \underline{4}. If the entry $a_{13}$ changes from $-3$ to $-10$, then   two hyperplanes in figure \ref{figure_13} are crossed, namely those of equations $a_{12}+a_{43}=a_{42}+a_{13}$ and   $a_{13}+a_{24}=a_{23}+a_{14}$, and the class of the tropical tetrahedron changes accordingly.
\end{ex}

Let us finish with a  remark.
Each configuration at $\underline{123}$ and at  $\underline{4}$ are compatible, except for  $(5.5.5)$  left and $(5.5.5)$ right. This follows from all the examples above, together with theorems \ref{thm:left_left}, \ref{thm:left_right} and \ref{thm:right_right},  using symmetry and changes of variables.

%
%
\begin{figure}[h]
 \centering
  \includegraphics[keepaspectratio,width=10cm]{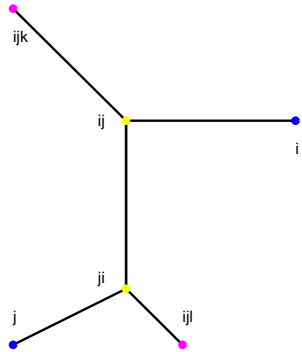}\\
  \caption{Extremal points in the neighborhood of points $\underline{ij}$ and $\underline{ji}$.}
  \label{figure_16}
\end{figure}

\begin{exr}
\begin{itemize}
\item Compute $\tdist (\underline{ijk}, \underline{ij})$, for different $i,j,k$; see figure\ref{figure_16}.
\item  Compare $\spann((A+A')/2)$ with $\spann(A)$ and $\spann(A')$, for different matrices $A,A'$ in this paper.
\end{itemize}
\end{exr}


\section*{Acknowledgements}  We are grateful to our colleague Raquel D\'{\i}az for some very useful questions and to S. Sergeev for drawing our attention to Kleene stars.

We have developed several programs in MATLAB for our tropical computations.  Factorizations of classical determinants have been done with MAPLE.


\centerline{\footnotesize{A. Jim\'{e}nez. Facultad de Matem\'{a}ticas. Universidad Complutense. Madrid. Spain. \texttt{adri\_yakuza@hotmail.com}}}
\centerline{\footnotesize{M. J. de la Puente. Dpto. de Algebra. Facultad de Matem\'{a}ticas. Universidad Complutense. Madrid. Spain. \texttt{mpuente@mat.ucm.es}}}
\end{document}